%% file: main.tex
\documentclass[12pt]{amsart}

\usepackage{amsfonts,amsmath,amsthm,latexsym,amssymb,subfigure,calc,mathtools}
\usepackage{graphicx,empheq,enumerate,epsfig}
\usepackage{tikz,color}
\usepackage{verbatim}
\usepackage{tikz-cd}
\usetikzlibrary{matrix, calc}
\usetikzlibrary{arrows,decorations.pathmorphing,backgrounds,positioning,fit}

\numberwithin{equation}{section}

\newtheorem{theorem}[equation]{Theorem}
\newtheorem{lemma}[equation]{Lemma}
\newtheorem{corollary}[equation]{Corollary}
\newtheorem{proposition}[equation]{Proposition}

\theoremstyle{definition}
\newtheorem{definition}[equation]{Definition}
\newtheorem{remark}[equation]{Remark}
\newtheorem{example}[equation]{Example}

\newcommand{\LE}{%
\hbox{%
\vbox{\hrule width 0.35em height 0.04 em}%
\vbox{\offinterlineskip%
\hbox{\kern -0.02em\vrule height 0.65em width 0.04em \hspace{0.1em}}%
}%
}%
}

\newcommand{\la}{\lambda}
\newcommand{\T}{\mathcal{T}}
\newcommand{\pt}{\mathcal{PT}}
\newcommand{\bt}{\mathcal{BT}}
\newcommand{\sym}{\mathfrak{S}}

\DeclareMathOperator{\inv}{inv}

\DeclareMathOperator{\n}{neg}
\DeclareMathOperator{\crs}{cr}
\DeclareMathOperator{\al}{al}
\DeclareMathOperator{\fwex}{fwex}

\DeclareMathOperator{\row}{row}
\DeclareMathOperator{\col}{col}
\DeclareMathOperator{\diag}{diag}

\DeclareMathOperator{\wex}{wex}

\DeclareMathOperator{\so}{so}

\DeclareMathOperator{\zero}{zero}
\DeclareMathOperator{\one}{one}
\DeclareMathOperator{\two}{two}

\DeclareMathOperator{\drop}{drop}
\DeclareMathOperator{\cyc}{cyc}

\DeclareMathOperator{\dess}{dess}

\DeclareMathOperator{\zerorow}{zerorow}

\begin{document}
\title[Permutation statistics and {B}ruhat order in permutation tableaux of type $B$]{Permutation statistics and weak {B}ruhat order in permutation tableaux of type $B$}
\author{Soojin Cho}
\address{
 Department of Mathematics \\
 Ajou University\\
 Suwon  443-749, Korea}
\email{chosj@ajou.ac.kr}

\author{Kyoungsuk Park}
\address{
 Department of Mathematics \\
 Ajou University\\
 Suwon  443-749, Korea}
\email{bluemk00@ajou.ac.kr}

\thanks{This research was supported by Basic Science Research Program through the
National Research Foundation of Korea(NRF) funded by the Ministry of
Education(NRF2011-0012398).}



\begin{abstract}
  Many important statistics of signed permutations are realized in the corresponding permutation tableaux or bare tableaux of type $B$: Alignments, crossings and inversions of signed permutations are realized in the corresponding permutation tableaux of type $B$, and the cycles of signed permutations are understood in the corresponding bare tableaux of type $B$. This leads us to relate the number of alignments and crossings with other statistics of signed permutations and also to characterize the covering relation in weak Bruhat order on Coxeter system of type $B$ in terms of permutation tableaux of type $B$.
\end{abstract}

\maketitle

\input{intro.tex}
\input{preliminary.tex}
\input{PT.tex}
\input{bruhat.tex}
\input{BT.tex}

\bibliographystyle{amsplain}
\bibliography{PTB}

\end{document}

%% file: intro.tex
\section{Introduction}
Since the notion of \emph{permutation tableaux} was introduced in \cite{ES-LW}, there have been a great deal of work on the subject in many different directions. Many equivalent combinatorial objects were introduced for various purposes, new statistics were defined in the relation to the permutation statistics, and extensions in terms of Lie types were considered. Works in the connection to PASEP(Partially Asymmetric Exclusion Process) in statistical physics \cite{SC-LW,SC-LW-markov} and Askey-Wilson polynomials \cite{SC-RS-DS-LW} are most remarkable among many related works on permutation tableaux.

One of the main importances of permutation tableaux relies on the fact that they are in bijection with permutations and therefore provide another point of view to look at permutations in both types $A$ and $B$. Some statistics on (signed) permutations are very well understood in permutation tableaux: There are at least two kinds of permutation tableaux and each of them has its own advantage over certain permutation statistics, and statistics naturally defined on permutation tableaux make it possible to develop further combinatorial theories on Coxeter groups.

The \emph{inversion number} of a (signed) permutation is an important statistic, especially in the theory of Coxeter groups: Inversion number is the \emph{length} of a (signed) permutation and hence it plays an essential role to understand the Bruhat order on the Coxeter groups of type $A$ and $B$. However, there was no known way to understand \emph{inversions} directly in the permutation tableaux. In the present paper we prove a theorem to realize the inversion number directly from the permutation tableaux (Theorem~\ref{bruhatinv}). The proof of Theorem~\ref{bruhatinv} enables us to realize the covering relation in weak Bruhat order on Coxeter groups of both types $A$ and $B$ in terms of permutation tableaux (Theorem \ref{weakbruhatcover}). Inversion numbers are closely related to \emph{alignments} and \emph{crossings}, which are relatively new statistics introduced by S. Corteel in \cite{SC} for (signed) permutations, and we started our work by closely looking at the alignments and crossings in permutation tableaux of type $B$, that leads us to prove an equation that writes the sum of alignments and crossings in terms of weak excedances and negatives (Theorem~\ref{sumofalcr}). Theorem~\ref{sumofalcr} resolves a problem (open problem 3) posed by Corteel et al. in \cite{SC-MJV-JK}.
We also show that the number of \emph{cycles} of a signed permutation counts the number of certain $1$'s in the corresponding \emph{bare tableau} of type $B$ and construct the inverse of the zigzag map from the bare tableaux of type $B$ to signed permutations.
We make a remark that our works on permutation statistics have been known for the cases of type $A$ (see Proposition~\ref{proptypeA}), and we extend the results to cases of type $B$. The work of ours on Bruhat order in permutation tableaux is new in both types $A$ and $B$.

The present paper is organized as follows:
In Section~\ref{preliminary} we introduce some necessary notions and related known results associated with signed permutations and permutation tableaux of type $B$. In Section~\ref{pt} we do work to understand alignments and crossings of signed permutations in permutation tableaux of type $B$.
Section~\ref{bruhat} is devoted to the work to understand inversion numbers in terms of permutation tableaux and hence the weak Bruhat order. Finally, in Section~\ref{bt}, we work on bare tableaux of type $B$ to understand the number of cycles in signed permutations.


%% file: preliminary.tex
\section{Preliminaries}\label{preliminary}
We introduce necessary notations, terms, and summarize known combinatorial results on permutation tableaux of type $B$ (and hence type $A$) in this section.

For a positive integer $n$, we use $[n]$ for the set $\{1,\ldots,n\}$ and $[\pm n]$ for the set $\{\pm 1, \ldots, \pm n\}$. A \emph{signed permutation} $\sigma$ \emph{of length} $n$ is a permutation on $[\pm n]$ satisfying $\sigma(-i)=-\sigma(i)$ for $i\in[n]$. The group of signed permutations is denoted by $\sym_{n}^{B}$, and the subgroup of (ordinary) permutations in $\sym_{n}^{B}$, which send positive numbers to positive numbers, is denoted by $\sym_{n}$.

We use two ways to write $\sigma\in\sym_{n}^{B}$: In \emph{one line notation} we write $\sigma=\sigma_{1} , \cdots , \sigma_{n}$ to mean that $\sigma(i)=\sigma_{i}$ for $i\in[n]$. In \emph{cycle notation} we write $\sigma=(a_{1,1} , \cdots , a_{1,k_{1}})\cdots(a_{m,1} , \cdots , a_{m,k_{m}})$ where $\{|a_{1,1}|,\ldots,|a_{1,k_{1}}|\}$, $\ldots$, $\{|a_{m,1}|,\ldots,|a_{m,k_{m}}|\}$ are mutually disjoint subsets of $[n]$, to mean that $\sigma(|a_{i,j}|)=a_{i,j+1}$ with the convention $a_{i,k_{i}+1}=a_{i,1}$ for each $i$. If $\sum_i k_i =n$ then the given cycle notation is called the \emph{full} cycle notation. Note that a single cycle $(a)\in\sym_{n}^{B}$ is the signed permutation $\sigma$ such that $\sigma(|a|)=a$ fixing all other elements. Hence, if $a>0$ then $(a)$ is the identity permutation and we usually omit $(a), a>0$, for a cycle notation of a signed permutation.  For example, $\sigma=4 , -2 , 1 , -3\in \sym_{4}^{B}$ means that $\sigma(1)=4$, $\sigma(2)=-2$, $\sigma(3)=1$, and $\sigma(4)=-3$, and a cycle notation for $\sigma$ is $\sigma=(1 , 4 , -3)(-2)$ and this is the full cycle notation of $\sigma$. Moreover, $\tau=-1,2,3\in \sym_{3}^{B}$ is the signed permutation such that $\tau(1)=-1$, $\tau(2)=2$, and $\tau(3)=3$, and $(-1)$ is a cycle notation and $(-1)(2)(3)$ is the full cycle notation for $\tau$.

We first give definitions of some statistics on signed permutations.

\begin{definition}
  Let $\sigma\in\sym_{n}^{B}$.
  \begin{itemize}
    \item $i\in[n]$ is a \emph{weak excedance} of $\sigma$ if $\sigma(i)\ge i$. The number of weak excedances of $\sigma$ is denoted by $\wex(\sigma)$.
    \item $i\in[n]$ is a \emph{drop} of $\sigma$ if $\sigma(i)<i$. The number of drops of $\sigma$ is denoted by $\drop(\sigma)$.
    \item The number of negative integers in the one line notation of $\sigma$ is denoted by $\n(\sigma)$.
    \item The number of cycles in the full cycle notation of $\sigma$ is denoted by $\cyc(\sigma)$.
    \item $\fwex(\sigma):=2 \wex(\sigma) + \n(\sigma)$.
  \end{itemize}
\end{definition}

For $\sigma=4,-2,1,-3\in \sym_{4}^{B}$, $\wex(\sigma)=1$, $\drop(\sigma)=3$,
 $\n(\sigma)=2$, $\cyc(\sigma)=2$, and $\fwex(\sigma)=4$, and for $\tau=-1,2,3\in\sym_{3}^{B}$, $\wex(\tau)=2$, $\drop(\tau)=1$, $\n(\tau)=1$, $\cyc(\tau)=3$, and $\fwex(\tau)=5$.
Note that $\wex(\sigma)+\drop(\sigma)=n$ for all $\sigma\in\sym_{n}^{B}$.

Corteel introduced \emph{alignments} and \emph{crossings} on permutations in \cite{SC}, and they were also defined for signed permutations in \cite{SC-MJV-JK}. We define alignments for signed permutations in a different way from the ones in \cite{SC-MJV-JK} for our arguments, while we use the same definition of crossings for signed permutations as in \cite{SC-MJV-JK}.

\begin{definition}\label{def_al}
  Let $\sigma\in\sym_{n}^{B}$.
  \begin{itemize}
    \item $A_{nest}(\sigma)=\{(i,j)\in [n]\times [n]\,|\,-i<-j<-\sigma(j)<-\sigma(i)\}$\\
        \indent\indent\indent\indent$\,\,\,\cup\,\{(i,j)\in [n]\times [n]\,|\,-i<j\le\sigma(j)<-\sigma(i)\}$\\
        \indent\indent\indent\indent$\,\,\,\cup\,\{(i,j)\in [n]\times [n]\,|\,i<j\le\sigma(j)<\sigma(i)\}$\\
        is the set of \emph{nestings} or \emph{alignments of nested type of $\sigma$}, and $\al_{nest}(\sigma)=|A_{nest}(\sigma)|$.
    \item $A_{EN}(\sigma)=\{(i,j)\in [n]\times[n]\,|\,-i<0<-\sigma(i)<\sigma(j)<j\}$\\
        \indent\indent\indent\indent$\,\cup\,\{(i,j)\in [n]\times[n]\,|\,i\le\sigma(i)<\sigma(j)<j\}$\\
        is the set of \emph{alignments of type $EN$ of $\sigma$}, and $\al_{EN}(\sigma)=|A_{EN}(\sigma)|$.
    \item $A_{NE}(\sigma)=\{(i,j)\in [n]\times [n]\,|\,\sigma(i)<i<j\le\sigma(j)\}$\\
        is the set of \emph{alignments of type $NE$ of $\sigma$}, and $\al_{NE}(\sigma)=|A_{NE}(\sigma)|$.
    \item $\al(\sigma)=\al_{nest}(\sigma)+\al_{EN}(\sigma)+\al_{NE}(\sigma)$ is the number of alignments of $\sigma$.
    \item $C(\sigma)=\{(i,j)\in[n]\times[n]\,|\,i<j\le\sigma(i)<\sigma(j)\}$\\
        \indent\indent\indent$\cup\,\{(i,j)\in [n]\times [n]\,|\,-i<-j<-\sigma(i)<-\sigma(j)\}$\\
        \indent\indent\indent$\cup\,\{(i,j)\in [n]\times [n]\,|\,-i<j\le -\sigma(i)<\sigma(j)\}$\\
        is the set of  \emph{crossings of $\sigma$}, and $\crs(\sigma)=|C(\sigma)|$ is the number of crossings of $\sigma$.
  \end{itemize}
\end{definition}

Our definition of alignments is consistent with the one in \cite{SC} for permutations in $\sym_{n}$.

\begin{example}\label{ex_alcr}
  Let $\sigma=-2,-4,5,3,1\in \sym_{5}^{B}$. Then $A_{nest}(\sigma)=\{(2,1),(5,4)\}$, $A_{EN}(\sigma)=\{(1,4)\}$, $A_{NE}(\sigma)=\{(1,3),(2,3)\}$, and $C(\sigma)=\{(5,2),(2,3)\}$. So, $\al(\sigma)=5$ and $\crs(\sigma)=2$. See Figure \ref{ex_alcrfig}.

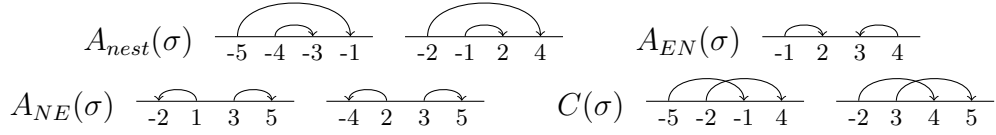
\begin{figure}[!ht]
  \begin{tikzpicture}\normalsize
    \node {$A_{nest}(\sigma)$};
  \end{tikzpicture}
  \begin{tikzpicture}\scriptsize
    \draw [-] (1.0,8.21) to (3.1,8.21);
    \node (a) at (1.3,8.0) {-5};
    \node (b) at (1.8,8.0) {-4};
    \node (c) at (2.3,8.0) {-3};
    \node (d) at (2.8,8.0) {-1};
    \path[->] (a) edge [out=90, in=90] (d);
    \path[->] (b) edge [out=90, in=90] (c);
  \end{tikzpicture}\quad
  \begin{tikzpicture}\scriptsize
    \draw [-] (1.0,8.21) to (3.1,8.21);
    \node (a) at (1.3,8.0) {-2};
    \node (b) at (1.8,8.0) {-1};
    \node (c) at (2.3,8.0) {2};
    \node (d) at (2.8,8.0) {4};
    \path[->] (a) edge [out=90, in=90] (d);
    \path[->] (b) edge [out=90, in=90] (c);
  \end{tikzpicture}\quad\quad
  \begin{tikzpicture}\normalsize
    \node {$A_{EN}(\sigma)$};
  \end{tikzpicture}
  \begin{tikzpicture}\scriptsize
    \draw [-] (1.0,8.21) to (3.1,8.21);
    \node (a) at (1.3,8.0) {-1};
    \node (b) at (1.8,8.0) {2};
    \node (c) at (2.3,8.0) {3};
    \node (d) at (2.8,8.0) {4};
    \path[->] (a) edge [out=90, in=90] (b);
    \path[->] (d) edge [out=90, in=90] (c);
  \end{tikzpicture}\\

  \begin{tikzpicture}\normalsize
    \node {$A_{NE}(\sigma)$};
  \end{tikzpicture}
  \begin{tikzpicture}\scriptsize
    \draw [-] (1.0,8.21) to (3.1,8.21);
    \node (a) at (1.3,8.0) {-2};
    \node (b) at (1.8,8.0) {1};
    \node (c) at (2.3,8.0) {3};
    \node (d) at (2.8,8.0) {5};
    \path[->] (b) edge [out=90, in=90] (a);
    \path[->] (c) edge [out=90, in=90] (d);
  \end{tikzpicture}\quad
  \begin{tikzpicture}\scriptsize
    \draw [-] (1.0,8.21) to (3.1,8.21);
    \node (a) at (1.3,8.0) {-4};
    \node (b) at (1.8,8.0) {2};
    \node (c) at (2.3,8.0) {3};
    \node (d) at (2.8,8.0) {5};
    \path[->] (b) edge [out=90, in=90] (a);
    \path[->] (c) edge [out=90, in=90] (d);
  \end{tikzpicture}\quad\quad
  \begin{tikzpicture}
    \node {$C(\sigma)$};
  \end{tikzpicture}
  \begin{tikzpicture}\scriptsize
    \draw [-] (1.0,8.21) to (3.1,8.21);
    \node (a) at (1.3,8.0) {-5};
    \node (b) at (1.8,8.0) {-2};
    \node (c) at (2.3,8.0) {-1};
    \node (d) at (2.8,8.0) {4};
    \path[->] (a) edge [out=90, in=90] (c);
    \path[->] (b) edge [out=90, in=90] (d);
  \end{tikzpicture}\quad
  \begin{tikzpicture}\scriptsize
    \draw [-] (1.0,8.21) to (3.1,8.21);
    \node (a) at (1.3,8.0) {-2};
    \node (b) at (1.8,8.0) {3};
    \node (c) at (2.3,8.0) {4};
    \node (d) at (2.8,8.0) {5};
    \path[->] (a) edge [out=90, in=90] (c);
    \path[->] (b) edge [out=90, in=90] (d);
  \end{tikzpicture}\vspace{-3mm}
  \caption{Alignments and crossings of $\sigma=-2,-4,5,3,1\in \sym_{5}^{B}$}\label{ex_alcrfig}
\end{figure}\normalsize
\end{example}

For two positive integers $k\leq n$, a \emph{$(k,n)$-diagram} is a left-justified diagram of boxes in a $k\times (n-k)$ rectangle with $\la_i$ boxes in the $i$th row, where $\la_1\geq \la_2 \geq \dots \geq \la_{k}\ge 0$. A \emph{shifted $(k,n)$-diagram} is a $(k,n)$-diagram together with the
stair shaped array of boxes added above, where the $j$th column (from the left) has $(n-k-j+1)$ additional boxes for $j\in[n-k]$. We call the (unique) $(k,n)$-diagram in a shifted $(k,n)$-diagram the \emph{$(k,n)$-subdiagram}, and the $(n-k)$ topmost boxes in a shifted $(k,n)$-diagram \emph{diagonals}. See Figure~\ref{young_shiftedyoung}.

Rows and columns of a $(k,n)$-diagram are labeled as follows: From the northeast corner to the southwest corner, follow the southeast border edges of the diagram and give labels $1,2,\ldots,n$ in order. If a south edge earned the label $i$ then the corresponding column is labeled by $i$, and if an east edge earned the label $j$ then the corresponding row is labeled by $j$. We put the labels on the left of each row and on the top of each column.  For a shifted $(k,n)$-diagram, label the rows and columns of the $(k,n)$-subdiagram with the numbers in $[n]$ first. Then label the remaining rows as follows; if the diagonal in a row is in the column labeled $i$, then the row is labeled by $-i$. We usually put the label on the left of each row and omit the column labels.
We use \verb"row"\,$i$, $i\in [\pm n]$, to denote the row labeled by $i$ and \verb"col"\,$j$ for the column labeled by $j$. See Figure~\ref{young_shiftedyoung}.

\tiny
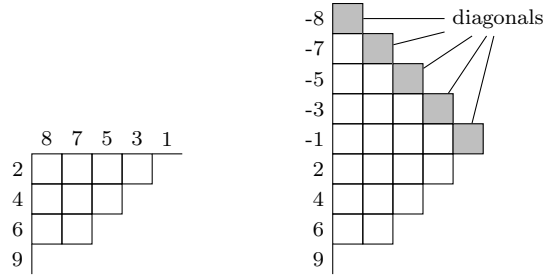
\begin{figure}[!ht]
  \centering
  \begin{tikzpicture}
    \tikzstyle{Element} = [draw, minimum width=4mm, minimum height=4mm, node distance=4mm, inner sep=0pt]
      \node [Element] [label=left:2, label=above:8] at (1.0,4.0) {};
      \node [Element] [label=left:4] at (1.0,3.6) {};
      \node [Element] [label=left:6] at (1.0,3.2) {};
      \draw[-] (0.8,3.0) to (0.8,2.6);
      \node at (0.61,2.8) {9};

      \node [Element] [label=above:7] at (1.4,4.0) {};
      \node [Element] at (1.4,3.6) {};
      \node [Element] at (1.4,3.2) {};

      \node [Element] [label=above:5] at (1.8,4.0) {};
      \node [Element] at (1.8,3.6) {};

      \node [Element] [label=above:3] at (2.2,4.0) {};

      \draw[-] (2.4,4.2) to (2.8,4.2);
      \node at (2.6,4.39) {1};
      \node (d) at (7.0,6.0) {diagonals};
      \node (d1) [Element,fill=lightgray] [label=left:-8] at (5.0,6.0) {};
      \node [Element] [label=left:-7] at (5.0,5.6) {};
      \node [Element] [label=left:-5] at (5.0,5.2) {};
      \node [Element] [label=left:-3] at (5.0,4.8) {};
      \node [Element] [label=left:-1] at (5.0,4.4) {};
      \node [Element] [label=left:2] at (5.0,4.0) {};
      \node [Element] [label=left:4] at (5.0,3.6) {};
      \node [Element] [label=left:6] at (5.0,3.2) {};
      \draw[-] (4.8,3.0) to (4.8,2.6);
      \node at (4.61,2.8) {9};

      \node (d2) [Element,fill=lightgray] at (5.4,5.6) {};
      \node [Element] at (5.4,5.2) {};
      \node [Element] at (5.4,4.8) {};
      \node [Element] at (5.4,4.4) {};
      \node [Element] at (5.4,4.0) {};
      \node [Element] at (5.4,3.6) {};
      \node [Element] at (5.4,3.2) {};

      \node (d3) [Element,fill=lightgray] at (5.8,5.2) {};
      \node [Element] at (5.8,4.8) {};
      \node [Element] at (5.8,4.4) {};
      \node [Element] at (5.8,4.0) {};
      \node [Element] at (5.8,3.6) {};

      \node (d4) [Element,fill=lightgray] at (6.2,4.8) {};
      \node [Element] at (6.2,4.4) {};
      \node [Element] at (6.2,4.0) {};

      \node (d5) [Element,fill=lightgray] at (6.6,4.4) {};

      \draw [-] (d) to (d1);
      \draw [-] (d) to (d2);
      \draw [-] (d) to (d3);
      \draw [-] (d) to (d4);
      \draw [-] (d) to (d5);
  \end{tikzpicture}\vspace{-3mm}
  \caption{A $(4,9)$-diagram and a shifted $(4,9)$-diagram with labelings}\label{young_shiftedyoung}
\end{figure}
\normalsize

A \emph{permutation tableau of type $B$} is a $(0,1)$-filling of a shifted $(k,n)$-diagram that satisfies the following conditions:
\begin{itemize}
  \item[(1)] Every column has at least one $1$.
  \item[(2)] (\emph{$\LE$-condition} or \emph{$1$-hinge condition}) A box with a $1$ above it in the same column and a $1$ to the left in the same row has a $1$.
  \item[(3)] If a diagonal has a $0$, then there is no $1$ in the same row.
\end{itemize}

A \emph{bare tableau of type $B$} is a $(0,1)$-filling of a shifted $(k,n)$-diagram that satisfies the following conditions:
\begin{itemize}
  \item[(1)] Every column has at least one $1$.
  \item[(2)] (\emph{$0$-hinge condition}) A box with a $1$ above it in the same column and a $1$ to the left in the same row has a $0$.
  \item[(3)] If a diagonal has a $0$, then there is no $1$ in the same row.
\end{itemize}

Permutation tableaux (or bare tableaux) \emph{of type $A$} are permutation tableaux (or bare tableaux, respectively) of type $B$ such that the diagonals are filled with all $0$'s. The \emph{length} of any filling of a (shifted) $(k, n)$-diagram is $n$.

\tiny
\begin{figure}[!ht]
  \centering
  \begin{tikzpicture}
    \tikzstyle{Element} = [draw, minimum width=4mm, minimum height=4mm, node distance=4mm, inner sep=0pt]
      \node [Element] [label=left:-8] at (5.0,5.6) {1};
      \node [Element] [label=left:-6] at (5.0,5.2) {0};
      \node [Element] [label=left:-5] at (5.0,4.8) {1};
      \node [Element] [label=left:-3] at (5.0,4.4) {1};
      \node [Element] [label=left:1] at (5.0,4.0) {0};
      \node [Element] [label=left:2] at (5.0,3.6) {1};
      \node [Element] [label=left:4] at (5.0,3.2) {0};
      \node [Element] [label=left:7] at (5.0,2.8) {1};

      \node [Element] at (5.4,5.2) {0};
      \node [Element] at (5.4,4.8) {1};
      \node [Element] at (5.4,4.4) {1};
      \node [Element] at (5.4,4.0) {0};
      \node [Element] at (5.4,3.6) {1};
      \node [Element] at (5.4,3.2) {1};

      \node [Element] at (5.8,4.8) {1};
      \node [Element] at (5.8,4.4) {1};
      \node [Element] at (5.8,4.0) {0};
      \node [Element] at (5.8,3.6) {1};
      \node [Element] at (5.8,3.2) {1};

      \node [Element] at (6.2,4.4) {1};
      \node [Element] at (6.2,4.0) {1};
      \node [Element] at (6.2,3.6) {1};
  \end{tikzpicture}\quad\quad
  \begin{tikzpicture}
    \tikzstyle{Element} = [draw, minimum width=4mm, minimum height=4mm, node distance=4mm, inner sep=0pt]
      \node [Element] [label=left:-8] at (5.0,5.6) {1};
      \node [Element] [label=left:-6] at (5.0,5.2) {0};
      \node [Element] [label=left:-5] at (5.0,4.8) {1};
      \node [Element] [label=left:-3] at (5.0,4.4) {0};
      \node [Element] [label=left:1] at (5.0,4.0) {1};
      \node [Element] [label=left:2] at (5.0,3.6) {1};
      \node [Element] [label=left:4] at (5.0,3.2) {0};
      \node [Element] [label=left:7] at (5.0,2.8) {1};

      \node [Element] at (5.4,5.2) {0};
      \node [Element] at (5.4,4.8) {0};
      \node [Element] at (5.4,4.4) {0};
      \node [Element] at (5.4,4.0) {1};
      \node [Element] at (5.4,3.6) {0};
      \node [Element] at (5.4,3.2) {1};

      \node [Element] at (5.8,4.8) {1};
      \node [Element] at (5.8,4.4) {1};
      \node [Element] at (5.8,4.0) {0};
      \node [Element] at (5.8,3.6) {0};
      \node [Element] at (5.8,3.2) {0};

      \node [Element] at (6.2,4.4) {1};
      \node [Element] at (6.2,4.0) {0};
      \node [Element] at (6.2,3.6) {0};
  \end{tikzpicture}\vspace{-3mm}
  \caption{A permutation tableau and a bare tableau of type $B$ of length $8$}\label{ptb_btb}
\end{figure}
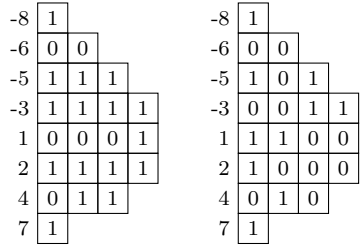\normalsize

The set of permutation tableaux of type $B$ (or type $A$) of length $n$ is denoted by $\pt_{n}^{B}$ (or $\pt_{n}$, respectively) and the set of bare tableaux of type $B$ (or type $A$, respectively) of length $n$ is denoted by $\bt_{n}^{B}$ (or $\bt_{n}$, respectively).
It is well known that permutation tableaux of type $B$ are in bijection with bare tableaux of type $B$. (\cite{SC-JK})

As in \cite{ES-LW}, we can represent a permutation tableau of type $B$ in a different way; fill in the boxes outside the diagram (within the $k\times (n-k)$ rectangle) with $2$'s; see Figure~\ref{pt_twos}. We call this representation the \emph{extended} representation of $\T$. Then the boxes with $2$'s bijectively correspond to the alignments of type $NE$ of the corresponding permutation for type $A$ (Theorem 14 in \cite{ES-LW}), which will be proved to be true for the case of type $B$ in Proposition~\ref{alignments}.

\begin{figure}[!ht]
  \centering
  \begin{tikzpicture}\tiny
    \tikzstyle{Element} = [draw, minimum width=4mm, minimum height=4mm, node distance=4mm, inner sep=0pt]
      \node [Element] [label=left:-8] at (5.0,5.6) {0};
      \node [Element] [label=left:-6] at (5.0,5.2) {1};
      \node [Element] [label=left:-5] at (5.0,4.8) {0};
      \node [Element] [label=left:-3] at (5.0,4.4) {0};
      \node [Element] [label=left:-1] at (5.0,4.0) {1};
      \node [Element] [label=left:2] at (5.0,3.6) {0};
      \node [Element] [label=left:4] at (5.0,3.2) {1};
      \node [Element] [label=left:7] at (5.0,2.8) {0};
      \draw[-] (4.8,2.6) to (4.8,2.2);
      \node at (4.61,2.4) {9};

      \node [Element] at (5.4,5.2) {1};
      \node [Element] at (5.4,4.8) {0};
      \node [Element] at (5.4,4.4) {0};
      \node [Element] at (5.4,4.0) {1};
      \node [Element] at (5.4,3.6) {1};
      \node [Element] at (5.4,3.2) {1};

      \node [Element] at (5.8,4.8) {0};
      \node [Element] at (5.8,4.4) {1};
      \node [Element] at (5.8,4.0) {1};
      \node [Element] at (5.8,3.6) {1};
      \node [Element] at (5.8,3.2) {1};

      \node [Element] at (6.2,4.4) {0};
      \node [Element] at (6.2,4.0) {1};
      \node [Element] at (6.2,3.6) {1};

      \node [Element] at (6.6,4) {1};

      \normalsize
      \node at (7.4,4.2) {$=$};
  \end{tikzpicture}
  \begin{tikzpicture}\tiny
    \tikzstyle{Element} = [draw, minimum width=4mm, minimum height=4mm, node distance=4mm, inner sep=0pt]
      \node [Element] [label=left:-8] at (5.0,5.6) {0};
      \node [Element] [label=left:-6] at (5.0,5.2) {1};
      \node [Element] [label=left:-5] at (5.0,4.8) {0};
      \node [Element] [label=left:-3] at (5.0,4.4) {0};
      \node [Element] [label=left:-1] at (5.0,4.0) {1};
      \node [Element] [label=left:2] at (5.0,3.6) {0};
      \node [Element] [label=left:4] at (5.0,3.2) {1};
      \node [Element] [label=left:7] at (5.0,2.8) {0};
      \node [Element,fill=lightgray] [label=left:9] at (5.0,2.4) {2};

      \node [Element] at (5.4,5.2) {1};
      \node [Element] at (5.4,4.8) {0};
      \node [Element] at (5.4,4.4) {0};
      \node [Element] at (5.4,4.0) {1};
      \node [Element] at (5.4,3.6) {1};
      \node [Element] at (5.4,3.2) {1};
      \node [Element,fill=lightgray] at (5.4,2.8) {2};
      \node [Element,fill=lightgray] at (5.4,2.4) {2};

      \node [Element] at (5.8,4.8) {0};
      \node [Element] at (5.8,4.4) {1};
      \node [Element] at (5.8,4.0) {1};
      \node [Element] at (5.8,3.6) {1};
      \node [Element] at (5.8,3.2) {1};
      \node [Element,fill=lightgray] at (5.8,2.8) {2};
      \node [Element,fill=lightgray] at (5.8,2.4) {2};

      \node [Element] at (6.2,4.4) {0};
      \node [Element] at (6.2,4.0) {1};
      \node [Element] at (6.2,3.6) {1};
      \node [Element,fill=lightgray] at (6.2,3.2) {2};
      \node [Element,fill=lightgray] at (6.2,2.8) {2};
      \node [Element,fill=lightgray] at (6.2,2.4) {2};

      \node [Element] at (6.6,4) {1};
      \node [Element,fill=lightgray] at (6.6,3.6) {2};
      \node [Element,fill=lightgray] at (6.6,3.2) {2};
      \node [Element,fill=lightgray] at (6.6,2.8) {2};
      \node [Element,fill=lightgray] at (6.6,2.4) {2};
  \end{tikzpicture}\vspace{-3mm}
  \caption{The extended representation of a permutation tableau of type $B$}\label{pt_twos}
\end{figure}
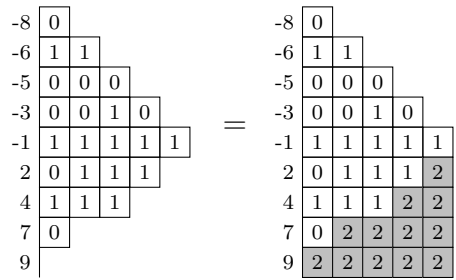\normalsize

E. Steingr\'{i}msson and L. Williams defined a bijective zigzag map $\Phi:\pt_{n}\rightarrow\sym_{n}$ in \cite{ES-LW} and it was extended to $\zeta:\pt_{n}^{B}\rightarrow\sym_{n}^{B}$ by Corteel and J. Kim in \cite{SC-JK}. A bijective zigzag map $\Phi_{bare}:\bt_{n}\rightarrow\sym_{n}$ was defined (\cite{AB,PN,SC-JK}) and it can be naturally extended to a bijection $\zeta_{bare}:\bt_{n}^{B}\rightarrow\sym_{n}^{B}$.

For a permutation tableau or a bare tableau $\T$ of type $B$,
a \emph{zigzag path from} $\verb"row"\,j$ (or $\verb"col"\,i$) in $\T$, is the path starting from the left of  $\verb"row"\,j$ (or the top of $\verb"col"\,i$, respectively), moving east (or south, respectively) until it meets the right of a row or the bottom of a column where it changes the direction to either south or east whenever it meets a $1$. Then the \emph{zigzag map} $\zeta: \pt_{n}^{B}\rightarrow\sym_{n}^{B}$ is defined as follows: For $\T\in\pt_{n}^{B}$ and $i\in[n]$,
\begin{itemize}
  \item[(1)] if $i$ is the label of a row of $\T$, then $\zeta(\T)(i)$ is the label of the row or the column that the zigzag path from $\verb"row"\,i$ ends,
  \item[(2)] if $-i$ is the label of a row of $\T$ and $\verb"row"\,(-i)$ has $0$ in its diagonal, then $\zeta(\T)(i)$ is the label of the row or the column that the zigzag path from $\verb"col"\,i$ ends, and
  \item[(3)] if $-i$ is the label of a row of $\T$ and $\verb"row"\,(-i)$ has $1$ in its diagonal, then $\zeta(\T)(-i)$ is the label of the row or the column that the zigzag path from $\verb"row"\,(-i)$ ends.
\end{itemize}

The zigzag map $\zeta_{bare}:\bt_{n}^{B}\rightarrow\sym_{n}^{B}$ is defined in the same way. We construct $\zeta_{bare}^{-1}$ in Section~\ref{bt} which shows that $\zeta_{bare}$ is bijective. Zigzag maps $\zeta$ and $\zeta_{bare}$ are described in Figure~\ref{zigzagex}.

\tiny
\begin{figure}[!ht]
  \centering
  \begin{tikzpicture}
    \tikzstyle{Element} = [draw, minimum width=4mm, minimum height=4mm, node distance=4mm, inner sep=0pt]
    \tikzstyle{Vertex} = [draw,circle, minimum size=3mm, inner sep=0pt]
      \node [Element] [label=left:-8] at (5.0,5.6) {\textcolor{gray}{1}};
      \node [Element] [label=left:-6] at (5.0,5.2) {\textcolor{gray}{0}};
      \node [Element] [label=left:-5] at (5.0,4.8) {\textcolor{gray}{1}};
      \node [Element] [label=left:-3] at (5.0,4.4) {\textcolor{gray}{1}};
      \node [Element] at (5.0,4.0) {\textcolor{gray}{0}};
      \node [Element] [label=left:2] at (5.0,3.6) {\textcolor{gray}{1}};
      \node [Element] [label=left:4] at (5.0,3.2) {\textcolor{gray}{0}};
      \node [Element] [label=left:7] at (5.0,2.8) {\textcolor{gray}{1}};

      \node [Element] at (5.4,5.2) {\textcolor{gray}{0}};
      \node [Element] at (5.4,4.8) {\textcolor{gray}{1}};
      \node [Element] at (5.4,4.4) {\textcolor{gray}{1}};
      \node [Element] at (5.4,4.0) {\textcolor{gray}{0}};
      \node [Element] at (5.4,3.6) {\textcolor{gray}{1}};
      \node [Element] at (5.4,3.2) {\textcolor{gray}{1}};

      \node [Element] at (5.8,4.8) {\textcolor{gray}{1}};
      \node [Element] at (5.8,4.4) {\textcolor{gray}{1}};
      \node [Element] at (5.8,4.0) {\textcolor{gray}{0}};
      \node [Element] at (5.8,3.6) {\textcolor{gray}{1}};
      \node [Element] at (5.8,3.2) {\textcolor{gray}{1}};

      \node [Element] at (6.2,4.4) {\textcolor{gray}{1}};
      \node [Element] at (6.2,4.0) {\textcolor{gray}{1}};
      \node [Element] at (6.2,3.6) {\textcolor{gray}{1}};

      \node [Vertex] (0) at (4.6,4.0) {\textbf{1}};
      \node [Vertex] (1) at (6.6,3.6) {\textbf{2}};
      \path [->] (0) edge (6.2,4.0);
      \path [->] (6.2,4.0) edge (6.2,3.6);
      \path [->] (6.2,3.6) edge (1);
  \end{tikzpicture}
  \begin{tikzpicture}
    \tikzstyle{Element} = [draw, minimum width=4mm, minimum height=4mm, node distance=4mm, inner sep=0pt]
    \tikzstyle{Vertex} = [draw,circle, minimum size=3mm, inner sep=0pt]
      \node [Element] [label=left:-8] at (5.0,5.6) {\textcolor{gray}{1}};
      \node [Element] [label=left:-6] at (5.0,5.2) {\textcolor{gray}{0}};
      \node [Element] [label=left:-5] at (5.0,4.8) {\textcolor{gray}{1}};
      \node [Element] [label=left:-3] at (5.0,4.4) {\textcolor{gray}{1}};
      \node [Element] [label=left:1] at (5.0,4.0) {\textcolor{gray}{0}};
      \node [Element] [label=left:2] at (5.0,3.6) {\textcolor{gray}{1}};
      \node [Element] [label=left:4] at (5.0,3.2) {\textcolor{gray}{0}};
      \node [Element] [label=left:7] at (5.0,2.8) {\textcolor{gray}{1}};

      \node [Element] at (5.4,5.2) {\textcolor{red}{0}};
      \node [Element] at (5.4,4.8) {\textcolor{gray}{1}};
      \node [Element] at (5.4,4.4) {\textcolor{gray}{1}};
      \node [Element] at (5.4,4.0) {\textcolor{gray}{0}};
      \node [Element] at (5.4,3.6) {\textcolor{gray}{1}};
      \node [Element] at (5.4,3.2) {\textcolor{gray}{1}};

      \node [Element] at (5.8,4.8) {\textcolor{gray}{1}};
      \node [Element] at (5.8,4.4) {\textcolor{gray}{1}};
      \node [Element] at (5.8,4.0) {\textcolor{gray}{0}};
      \node [Element] at (5.8,3.6) {\textcolor{gray}{1}};
      \node [Element] at (5.8,3.2) {\textcolor{gray}{1}};

      \node [Element] at (6.2,4.4) {\textcolor{gray}{1}};
      \node [Element] at (6.2,4.0) {\textcolor{gray}{1}};
      \node [Element] at (6.2,3.6) {\textcolor{gray}{1}};

      \node [Vertex] (0) at (5.4,5.6) {\textbf{6}};
      \node [Vertex] (1) at (6.6,4.0) {\textbf{1}};
      \path [->] (0) edge (5.4,4.8);
      \path [->] (5.4,4.8) edge (5.8,4.8);
      \path [->] (5.8,4.8) edge (5.8,4.4);
      \path [->] (5.8,4.4) edge (6.2,4.4);
      \path [->] (6.2,4.4) edge (6.2,4.0);
      \path [->] (6.2,4.0) edge (1);
  \end{tikzpicture}
  \begin{tikzpicture}
    \tikzstyle{Element} = [draw, minimum width=4mm, minimum height=4mm, node distance=4mm, inner sep=0pt]
    \tikzstyle{Vertex} = [draw,circle, minimum size=3mm, inner sep=0pt]
      \node [Element] at (5.0,5.6) {\textcolor{red}{1}};
      \node [Element] [label=left:-6] at (5.0,5.2) {\textcolor{gray}{0}};
      \node [Element] [label=left:-5] at (5.0,4.8) {\textcolor{gray}{1}};
      \node [Element] [label=left:-3] at (5.0,4.4) {\textcolor{gray}{1}};
      \node [Element] [label=left:1] at (5.0,4.0) {\textcolor{gray}{0}};
      \node [Element] [label=left:2] at (5.0,3.6) {\textcolor{gray}{1}};
      \node [Element] [label=left:4] at (5.0,3.2) {\textcolor{gray}{0}};
      \node [Element] [label=left:7] at (5.0,2.8) {\textcolor{gray}{1}};

      \node [Element] at (5.4,5.2) {\textcolor{gray}{0}};
      \node [Element] at (5.4,4.8) {\textcolor{gray}{1}};
      \node [Element] at (5.4,4.4) {\textcolor{gray}{1}};
      \node [Element] at (5.4,4.0) {\textcolor{gray}{0}};
      \node [Element] at (5.4,3.6) {\textcolor{gray}{1}};
      \node [Element] at (5.4,3.2) {\textcolor{gray}{1}};

      \node [Element] at (5.8,4.8) {\textcolor{gray}{1}};
      \node [Element] at (5.8,4.4) {\textcolor{gray}{1}};
      \node [Element] at (5.8,4.0) {\textcolor{gray}{0}};
      \node [Element] at (5.8,3.6) {\textcolor{gray}{1}};
      \node [Element] at (5.8,3.2) {\textcolor{gray}{1}};

      \node [Element] at (6.2,4.4) {\textcolor{gray}{1}};
      \node [Element] at (6.2,4.0) {\textcolor{gray}{1}};
      \node [Element] at (6.2,3.6) {\textcolor{gray}{1}};

      \node [Vertex] (0) at (4.6,5.6) {\textbf{-8}};
      \node [Vertex] (1) at (6.2,3.2) {\textbf{3}};
      \path [->] (0) edge (5.0,5.6);
      \path [->] (5.0,5.6) edge (5.0,4.8);
      \path [->] (5.0,4.8) edge (5.4,4.8);
      \path [->] (5.4,4.8) edge (5.4,4.4);
      \path [->] (5.4,4.4) edge (5.8,4.4);
      \path [->] (5.8,4.4) edge (5.8,3.6);
      \path [->] (5.8,3.6) edge (6.2,3.6);
      \path [->] (6.2,3.6) edge (1);
  \end{tikzpicture}\quad\quad\quad\quad\quad\quad\quad
  \begin{tikzpicture}
    \tikzstyle{Element} = [draw, minimum width=4mm, minimum height=4mm, node distance=4mm, inner sep=0pt]
    \tikzstyle{Vertex} = [draw,circle, minimum size=3mm, inner sep=0pt]
      \node [Element] [label=left:-8] at (5.0,5.6) {\textcolor{gray}{1}};
      \node [Element] [label=left:-6] at (5.0,5.2) {\textcolor{gray}{0}};
      \node [Element] [label=left:-5] at (5.0,4.8) {\textcolor{gray}{1}};
      \node [Element] [label=left:-3] at (5.0,4.4) {\textcolor{gray}{0}};
      \node [Element] at (5.0,4.0) {\textcolor{gray}{1}};
      \node [Element] [label=left:2] at (5.0,3.6) {\textcolor{gray}{1}};
      \node [Element] [label=left:4] at (5.0,3.2) {\textcolor{gray}{0}};
      \node [Element] [label=left:7] at (5.0,2.8) {\textcolor{gray}{1}};

      \node [Element] at (5.4,5.2) {\textcolor{gray}{0}};
      \node [Element] at (5.4,4.8) {\textcolor{gray}{0}};
      \node [Element] at (5.4,4.4) {\textcolor{gray}{0}};
      \node [Element] at (5.4,4.0) {\textcolor{gray}{1}};
      \node [Element] at (5.4,3.6) {\textcolor{gray}{0}};
      \node [Element] at (5.4,3.2) {\textcolor{gray}{1}};

      \node [Element] at (5.8,4.8) {\textcolor{gray}{1}};
      \node [Element] at (5.8,4.4) {\textcolor{gray}{1}};
      \node [Element] at (5.8,4.0) {\textcolor{gray}{0}};
      \node [Element] at (5.8,3.6) {\textcolor{gray}{0}};
      \node [Element] at (5.8,3.2) {\textcolor{gray}{0}};

      \node [Element] at (6.2,4.4) {\textcolor{gray}{1}};
      \node [Element] at (6.2,4.0) {\textcolor{gray}{0}};
      \node [Element] at (6.2,3.6) {\textcolor{gray}{0}};

      \node [Vertex] (0) at (4.6,4.0) {\textbf{1}};
      \node [Vertex] (1) at (6.6,3.6) {\textbf{2}};
      \path [->] (0) edge (5.0,4.0);
      \path [->] (5.0,4.0) edge (5.0,3.6);
      \path [->] (5.0,3.6) edge (1);
  \end{tikzpicture}
  \begin{tikzpicture}
    \tikzstyle{Element} = [draw, minimum width=4mm, minimum height=4mm, node distance=4mm, inner sep=0pt]
    \tikzstyle{Vertex} = [draw,circle, minimum size=3mm, inner sep=0pt]
      \node [Element] [label=left:-8] at (5.0,5.6) {\textcolor{gray}{1}};
      \node [Element] [label=left:-6] at (5.0,5.2) {\textcolor{gray}{0}};
      \node [Element] [label=left:-5] at (5.0,4.8) {\textcolor{gray}{1}};
      \node [Element] [label=left:-3] at (5.0,4.4) {\textcolor{gray}{0}};
      \node [Element] [label=left:1] at (5.0,4.0) {\textcolor{gray}{1}};
      \node [Element] [label=left:2] at (5.0,3.6) {\textcolor{gray}{1}};
      \node [Element] [label=left:4] at (5.0,3.2) {\textcolor{gray}{0}};
      \node [Element] [label=left:7] at (5.0,2.8) {\textcolor{gray}{1}};

      \node [Element] at (5.4,5.2) {\textcolor{red}{0}};
      \node [Element] at (5.4,4.8) {\textcolor{gray}{0}};
      \node [Element] at (5.4,4.4) {\textcolor{gray}{0}};
      \node [Element] at (5.4,4.0) {\textcolor{gray}{1}};
      \node [Element] at (5.4,3.6) {\textcolor{gray}{0}};
      \node [Element] at (5.4,3.2) {\textcolor{gray}{1}};

      \node [Element] at (5.8,4.8) {\textcolor{gray}{1}};
      \node [Element] at (5.8,4.4) {\textcolor{gray}{1}};
      \node [Element] at (5.8,4.0) {\textcolor{gray}{0}};
      \node [Element] at (5.8,3.6) {\textcolor{gray}{0}};
      \node [Element] at (5.8,3.2) {\textcolor{gray}{0}};

      \node [Element] at (6.2,4.4) {\textcolor{gray}{1}};
      \node [Element] at (6.2,4.0) {\textcolor{gray}{0}};
      \node [Element] at (6.2,3.6) {\textcolor{gray}{0}};

      \node [Vertex] (0) at (5.4,5.6) {\textbf{6}};
      \node [Vertex] (1) at (6.6,4.0) {\textbf{1}};
      \path [->] (0) edge (5.4,4.0);
      \path [->] (5.4,4.0) edge (1);
  \end{tikzpicture}
  \begin{tikzpicture}
    \tikzstyle{Element} = [draw, minimum width=4mm, minimum height=4mm, node distance=4mm, inner sep=0pt]
    \tikzstyle{Vertex} = [draw,circle, minimum size=3mm, inner sep=0pt]
      \node [Element] at (5.0,5.6) {\textcolor{red}{1}};
      \node [Element] [label=left:-6] at (5.0,5.2) {\textcolor{gray}{0}};
      \node [Element] [label=left:-5] at (5.0,4.8) {\textcolor{gray}{1}};
      \node [Element] [label=left:-3] at (5.0,4.4) {\textcolor{gray}{0}};
      \node [Element] [label=left:1] at (5.0,4.0) {\textcolor{gray}{1}};
      \node [Element] [label=left:2] at (5.0,3.6) {\textcolor{gray}{1}};
      \node [Element] [label=left:4] at (5.0,3.2) {\textcolor{gray}{0}};
      \node [Element] [label=left:7] at (5.0,2.8) {\textcolor{gray}{1}};

      \node [Element] at (5.4,5.2) {\textcolor{gray}{0}};
      \node [Element] at (5.4,4.8) {\textcolor{gray}{0}};
      \node [Element] at (5.4,4.4) {\textcolor{gray}{0}};
      \node [Element] at (5.4,4.0) {\textcolor{gray}{1}};
      \node [Element] at (5.4,3.6) {\textcolor{gray}{0}};
      \node [Element] at (5.4,3.2) {\textcolor{gray}{1}};

      \node [Element] at (5.8,4.8) {\textcolor{gray}{1}};
      \node [Element] at (5.8,4.4) {\textcolor{gray}{1}};
      \node [Element] at (5.8,4.0) {\textcolor{gray}{0}};
      \node [Element] at (5.8,3.6) {\textcolor{gray}{0}};
      \node [Element] at (5.8,3.2) {\textcolor{gray}{0}};

      \node [Element] at (6.2,4.4) {\textcolor{gray}{1}};
      \node [Element] at (6.2,4.0) {\textcolor{gray}{0}};
      \node [Element] at (6.2,3.6) {\textcolor{gray}{0}};

      \node [Vertex] (0) at (4.6,5.6) {\textbf{-8}};
      \node [Vertex] (1) at (6.2,3.2) {\textbf{3}};
      \path [->] (0) edge (5.0,5.6);
      \path [->] (5.0,5.6) edge (5.0,4.8);
      \path [->] (5.0,4.8) edge (5.8,4.8);
      \path [->] (5.8,4.8) edge (5.8,4.4);
      \path [->] (5.8,4.4) edge (6.2,4.4);
      \path [->] (6.2,4.4) edge (1);
  \end{tikzpicture}\normalsize
  \vspace{-3mm}
  \caption{Zigzag maps in permutation tableaux and bare tableaux of type $B$}\label{zigzagex}
\end{figure}
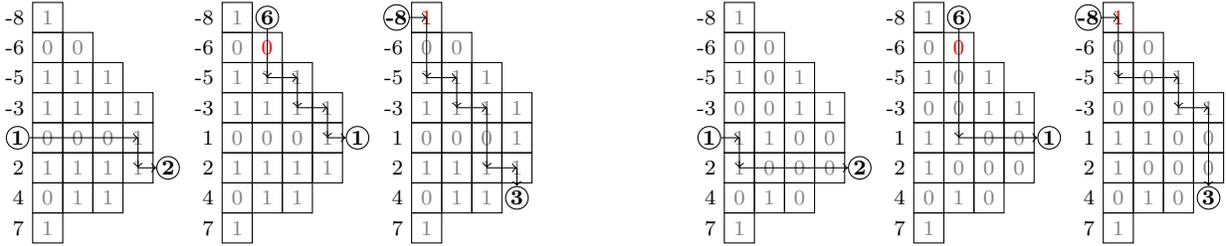\normalsize

There are three types of zeros in a permutation tableau. Note that there are two zigzag paths passing through a given $0$, one passing vertically and the other passing horizontally.

\begin{definition}\label{def_zero}
  Let $\T\in\pt_{n}^{B}$.
   \begin{itemize}
     \item A $0$ in $\T$ is \emph{of type $EE$} if two zigzag paths passing through the $0$ are both from rows of $\T$. The number of $0$'s of type $EE$ in $\T$ is denoted by $\zero_{EE}(\T)$.
     \item A $0$ in $\T$ is \emph{of type $NN$} if two zigzag paths passing through the $0$ are both from columns of $\T$. The number of $0$'s of type $NN$ in $\T$ is denoted by $\zero_{NN}(\T)$.
     \item A $0$ in $\T$ is \emph{of type $EN$} if it is not in a zero row labeled by a negative integer and a zigzag path passing through the $0$ is from a row and the other is from a column of $\T$. The number of $0$'s of type $EN$ in $\T$ is denoted by $\zero_{EN}(\T)$.
     \item A $0$ in $\T$ is \emph{non-typed} if it is not given a type of $EE$, $NN$, or $EN$.
     \item The number of typed $0$'s in $\T$ is denoted by $\zero(\T)$; $$\zero(\T)=\zero_{EE}(\T)+\zero_{NN}(\T)+\zero_{EN}(\T)\,.$$
     \item The numbers of $1$'s and $2$'s in the extended representation of $\T$ are denoted by  $\one(\T)$ and $\two(\T)$, respectively.
   \end{itemize}
\end{definition}

\begin{remark} Note that a $0$ in \emph{a zero row labeled by a negative integer} is either of type $EE$ or  non-typed, and non-typed $0$'s appear only in zero rows labeled by a negative integer. The definition of $\zero(\T)$ is consistent with the one for permutation tableaux of type $A$, which is given in \cite{ES-LW}.
\end{remark}

\begin{lemma}\label{non-typed zero}
The number of non-typed $0$'s in the zero row $\verb"row"\,(-i)$, $i\in[n]$, is the number of zero rows above $\verb"row"\,(-i)$, including $\verb"row"\,(-i)$ itself.
\end{lemma}
\begin{proof} Suppose that $\verb"col"\,\alpha$, $\alpha\ge i$, has $0$ in its diagonal, then
the zigzag path from $\verb"col"\,\alpha$ must pass a $0$ in $\verb"row"\,(-i)$, and such $0$ is non-typed. Conversely, each non-typed $0$ in $\verb"row"\,(-i)$ corresponds to a unique $0$ on a diagonal, if we follow (backwards) the zigzag path passing the non-typed $0$ vertically. This completes the proof.
\end{proof}

We introduce some statistics on permutation tableaux or bare tableaux of type $B$.
\begin{definition}
  Let $\T$ be a permutation tableau or a bare tableau of type $B$.
  \begin{itemize}
    \item A $1$ is an \emph{essential one} of $\T$, if it is the topmost $1$ in its column or the leftmost $1$ in its row. A $1$ is a \emph{doubly essential one} of $\T$ if it is both the topmost $1$ in its column and the leftmost $1$ in its row. The number of doubly essential ones of $\T$ is denoted by $\dess(\T)$.
    \item A $1$ is a \emph{superfluous one} of $\T$ if it is not the topmost $1$ in its column. The number of superfluous ones of $\T$ is denoted by $\so(\T)$.
    \item The number of rows of $\T$ with `positive labels', is denoted by $\row(\T)$.
    \item The number of zero rows (filled with all $0$'s) of $\T$ with `positive labels' is denoted by $\zerorow(\T)$.
  \end{itemize}
\end{definition}

\begin{remark} It is easy to see that each zero row of $\T\in\pt_{n}^{B}$ (or $\T\in\bt_{n}^{B}$) with positive label is corresponding to a fixed point of $\zeta(\T)$ (or $\zeta_{bare}(\T)$, respectively).
\end{remark}

Certain statistics on permutation tableaux and bare tableaux (of type $B$) correspond to statistics on (signed) permutations through the zigzag maps $\zeta$ and $\zeta_{bare}$:
\begin{proposition}[\cite{ES-LW,AB}]\label{proptypeA}
 For $\T\in\pt_{n}$ and $\T_{0}\in\bt_{n}$,
 \begin{enumerate}
   \item $\wex(\zeta(\T))=\row(\T)$,  $\crs(\zeta(\T))=\so(\T)$, $\al_{nest}(\zeta(\T))+\al_{EN}(\zeta(\T))=\zero(\T)$, and $\al_{NE}(\zeta(\T))=\two(\T)$,
   \item $\cyc(\zeta_{bare}(\T_{0}))=\dess(\T_{0})+\zerorow(\T_{0})$.
 \end{enumerate}
\end{proposition}

\begin{proposition}[\cite{SC-JK, SC-MJV-JK}]\label{propdiagram}
For $\T\in\pt_{n}^{B}$, $\wex(\zeta(\T))=\row(\T)$, and $\crs(\zeta(\T))=\so(\T)$.
\end{proposition}
We show that $\al_{nest}(\zeta(\T))+\al_{EN}(\zeta(\T))=\zero(\T)$ and $\al_{NE}(\zeta(\T))=\two(\T)$ also hold for a permutation tableau $\T$ of type $B$ in Section~\ref{pt}. Further, we prove that $\cyc(\zeta_{bare}(\T_{0}))=\dess(\T_{0})+\zerorow(\T_{0})$ also holds for a bare tableau $\T_0$  of type $B$ in Section~\ref{bt}. The following lemma is clear from the definition of $\zeta$ and $\zeta_{bare}$, and Proposition~\ref{propdiagram}.

\begin{lemma}\label{labeling}
  Let $\sigma=\zeta(\T) \in \sym_{n}^B$ for $\T\in\pt_{n}^{B}$ or $\sigma=\zeta_{bare}(\T)\in\sym_{n}^{B}$ for $\T\in\bt_{n}^{B}$. Then for $i\in[n]$,
  \begin{itemize}
    \item $i\in\wex(\sigma)$ if and only if $i$ is a row label of $\T$,
    \item $i\in\drop(\sigma)$ if and only if $(-i)$ is a row label of $\T$, and
    \item $i\in\n(\sigma)$ if and only if $(-i)$ is a row label of $\T$ and the diagonal of $\verb"row"\,(-i)$ is filled with $1$.
  \end{itemize}
\end{lemma}

%% file: PT.tex
\section{Alignments and crossings in permutation tableaux of type $B$}\label{pt}

In this section, we understand the alignments and crossings of signed permutations in their corresponding permutation tableaux of type $B$.

\begin{lemma}\label{cross_at_zero}
  Two zigzag paths in a permutation tableau of type $B$ crossing at a $0$ never cross again at another $0$.
\end{lemma}
\begin{proof}
  Suppose that they cross at another $0$ to the strictly southeast of the first $0$, then the $0$ must have a $1$ above and a $1$ to the left. This is a contradiction to the $\LE$-condition.
\end{proof}

The following proposition refines and generalizes the result in \cite{ES-LW} (see Theorem 14)  to type $B$. We use $EE$, $NN$, or $EN$ to denote $0$'s of corresponding types in permutation tableaux for convenience.

\begin{proposition}\label{alignments}
  Let $\T\in\pt_{n}^{B}$. Then,
    $$\al_{nest}(\zeta(\T))=\zero_{EE}(\T)+\zero_{NN}(\T),$$
    $$\al_{EN}(\zeta(\T))=\zero_{EN}(\T),\mbox{ and } \al_{NE}(\zeta(\T))=\two(\T).$$
\end{proposition}
\begin{proof}
Let $\T\in\pt_{n}^{B}$ and $\sigma=\zeta(\T)$.

Remember, from Definition~\ref{def_al}, that a pair $(i, j)\in [n]\times[n]$ is in $A_{nest}(\sigma)$ if and only if
$-i<-j<-\sigma(j)<-\sigma(i)$, $-i<j\le\sigma(j)<-\sigma(i)$, or $i<j\le\sigma(j)<\sigma(i)$.

Suppose that $-i<-j<-\sigma(j)<-\sigma(i)$ holds for a pair $(i, j)\in A_{nest}(\sigma)$. Then, $-i$ and $-j$ are row labels of $\T$ since $i, j\in\drop(\sigma)$, and there are three cases to be considered depending on the signs of $\sigma(j)$ and $\sigma(i)$;
 $-i<-j<-\sigma(j)<-\sigma(i)<0$, $-i<-j<-\sigma(j)<0<-\sigma(i)$, or $-i<-j<0<-\sigma(j)<-\sigma(i)$. Note that the diagonal entries of $\verb"row"\,(-i)$ and $\verb"row"\,(-j)$ are uniquely determined for each case, according to Lemma~\ref{labeling}. We consider two zigzag paths from $\verb"col"\,i$ and $\verb"col"\,j$ for the first case and two zigzag paths from $\verb"row"\,(-i)$ and $\verb"row"\,(-j)$ for the second and the third cases, and can check that two paths meet at a unique $0$ either of type $EE$ or of type $NN$ in each case; see Figure~\ref{pathee1}.

  \begin{figure}[!ht]
    \begin{tikzpicture}\tiny
      \draw[-] (1.0,8.0) to (1.0,4.1);
      \draw[-] (1.0,8.0) to (1.3,8.0);
      \draw[-] (1.3,8.0) to (1.3,7.7);
      \draw[-] (1.3,7.7) to (1.6,7.7);
      \draw[-] (1.6,7.7) to (1.6,7.4);
      \draw[-] (1.6,7.4) to (1.9,7.4);
      \draw[-] (1.9,7.4) to (1.9,7.1);
      \draw[-] (1.9,7.1) to (2.2,7.1);
      \draw[-] (2.2,7.1) to (2.2,6.8);
      \draw[-] (2.2,6.8) to (2.5,6.8);
      \draw[-] (2.5,6.8) to (2.5,6.5);
      \draw[-] (2.5,6.5) to (2.8,6.5);
      \draw[-] (2.8,6.5) to (2.8,6.2);
      \draw[-] (2.8,6.2) to (3.1,6.2);
      \draw[-] (3.1,6.2) to (3.1,5.0);
      \draw[-] (3.1,5.0) to (2.8,5.0);
      \draw[-] (2.8,5.0) to (2.8,4.4);
      \draw[-] (2.8,4.4) to (2.2,4.4);
      \draw[-] (2.2,4.4) to (2.2,4.1);
      \draw[-] (2.2,4.1) to (1.0,4.1);
      \tikzstyle{Element}=[draw,minimum width=3mm, minimum height=3mm, node distance=3mm, inner sep=0pt]
      \node [Element] [label=above:$i$] at (1.75,7.25) {0};
      \node [Element] [label=above:$j$] at (2.35,6.65) {0};
      \node [Element,fill=lightgray] at (2.35,6.0) {\texttt{NN}};
      \node at (0.7,7.25) {$-i$};
      \node at (0.7,6.65) {$-j$};
      \path[-] (1.75,7.4) edge (1.75,6.0);
      \path[-] (1.75,6.0) edge (2.8,6.0);
      \path[-] (2.8,6.0) edge (2.8,5.3);
      \path[->] (2.8,5.3) edge (3.1,5.3);
      \node at (3.45,5.3) {$\sigma(i)$};
      \path[dashed] (2.35,6.8) edge (2.35,5.1);
      \path[dashed] (2.35,5.1) edge (2.65,5.1);
      \path[dashed,->] (2.65,5.1) edge (2.65,4.4);
      \node at (2.65,4.2) {$\sigma(j)$};
      \scriptsize
      \node at (2.0,3.5) {$-i<-j<-\sigma(j)<-\sigma(i)<0$};
    \end{tikzpicture}\quad
    \begin{tikzpicture}\tiny
      \draw[-] (1.0,8.0) to (1.0,4.1);
      \draw[-] (1.0,8.0) to (1.3,8.0);
      \draw[-] (1.3,8.0) to (1.3,7.7);
      \draw[-] (1.3,7.7) to (1.6,7.7);
      \draw[-] (1.6,7.7) to (1.6,7.4);
      \draw[-] (1.6,7.4) to (1.9,7.4);
      \draw[-] (1.9,7.4) to (1.9,7.1);
      \draw[-] (1.9,7.1) to (2.2,7.1);
      \draw[-] (2.2,7.1) to (2.2,6.8);
      \draw[-] (2.2,6.8) to (2.5,6.8);
      \draw[-] (2.5,6.8) to (2.5,6.5);
      \draw[-] (2.5,6.5) to (2.8,6.5);
      \draw[-] (2.8,6.5) to (2.8,6.2);
      \draw[-] (2.8,6.2) to (3.1,6.2);
      \draw[-] (3.1,6.2) to (3.1,5.0);
      \draw[-] (3.1,5.0) to (2.8,5.0);
      \draw[-] (2.8,5.0) to (2.8,4.4);
      \draw[-] (2.8,4.4) to (2.2,4.4);
      \draw[-] (2.2,4.4) to (2.2,4.1);
      \draw[-] (2.2,4.1) to (1.0,4.1);
      \tikzstyle{Element}=[draw,minimum width=3mm, minimum height=3mm, node distance=3mm, inner sep=0pt]
      \node [Element] at (1.75,7.25) {1};
      \node [Element] at (2.35,6.65) {0};
      \node [Element,fill=lightgray] at (1.75,6.65) {\texttt{EE}};
      \node at (0.7,7.25) {$-i$};
      \node at (0.7,6.65) {$-j$};
      \path[-] (1.0,7.25) edge (1.75,7.25);
      \path[-] (1.75,7.25) edge (1.75,5.3);
      \path[->] (1.75,5.3) edge (3.1,5.3);
      \node at (3.55,5.3) {$\sigma(-i)$};
      \path[dashed,->] (1.0,6.65) edge (2.5,6.65);
      \scriptsize
      \node at (2.0,3.5) {$-i<-j<-\sigma(j)<0<-\sigma(i)$};
    \end{tikzpicture}\quad
    \begin{tikzpicture}\tiny
      \draw[-] (1.0,8.0) to (1.0,4.1);
      \draw[-] (1.0,8.0) to (1.3,8.0);
      \draw[-] (1.3,8.0) to (1.3,7.7);
      \draw[-] (1.3,7.7) to (1.6,7.7);
      \draw[-] (1.6,7.7) to (1.6,7.4);
      \draw[-] (1.6,7.4) to (1.9,7.4);
      \draw[-] (1.9,7.4) to (1.9,7.1);
      \draw[-] (1.9,7.1) to (2.2,7.1);
      \draw[-] (2.2,7.1) to (2.2,6.8);
      \draw[-] (2.2,6.8) to (2.5,6.8);
      \draw[-] (2.5,6.8) to (2.5,6.5);
      \draw[-] (2.5,6.5) to (2.8,6.5);
      \draw[-] (2.8,6.5) to (2.8,6.2);
      \draw[-] (2.8,6.2) to (3.1,6.2);
      \draw[-] (3.1,6.2) to (3.1,5.0);
      \draw[-] (3.1,5.0) to (2.8,5.0);
      \draw[-] (2.8,5.0) to (2.8,4.4);
      \draw[-] (2.8,4.4) to (2.2,4.4);
      \draw[-] (2.2,4.4) to (2.2,4.1);
      \draw[-] (2.2,4.1) to (1.0,4.1);
      \tikzstyle{Element}=[draw,minimum width=3mm, minimum height=3mm, node distance=3mm, inner sep=0pt]
      \node [Element] at (1.75,7.25) {1};
      \node [Element] at (2.35,6.65) {1};
      \node [Element,fill=lightgray] at (1.75,6.65) {\texttt{EE}};
      \node at (0.7,7.25) {$-i$};
      \node at (0.7,6.65) {$-j$};
      \path[-] (1.0,7.25) edge (1.75,7.25);
      \path[-] (1.75,7.25) edge (1.75,4.25);
      \path[->] (1.75,4.25) edge (2.2,4.25);
      \node at (2.65,4.25) {$\sigma(-i)$};
      \path[dashed] (1.0,6.65) edge (2.35,6.65);
      \path[dashed] (2.35,6.65) edge (2.35,5.3);
      \path[dashed,->] (2.35,5.3) edge (3.1,5.3);
      \node at (3.55,5.3) {$\sigma(-j)$};
      \scriptsize
      \node at (2.0,3.5) {$-i<-j<0<-\sigma(j)<-\sigma(i)$};
    \end{tikzpicture}\normalsize\vspace{-3mm}
    \caption{Two zigzag paths}
    \label{pathee1}
  \end{figure}

Suppose that $-i<j\le\sigma(j)<-\sigma(i)$ holds for a pair $(i, j)\in A_{nest}(\sigma)$.
Then, $-i$ and $j$ are labels of rows in $\T$ and we consider two zigzag paths from $\verb"row"\,(-i)$ and $\verb"row"\,j$; it is easy to check that they cross at a unique $0$ of type ${EE}$ by Definition \ref{def_zero} and  Lemma \ref{cross_at_zero}.
Suppose that $i<j\le\sigma(j)<\sigma(i)$ holds for a pair $(i, j)\in A_{nest}(\sigma)$. Then two zigzag paths from $\verb"row"\,i$ and $\verb"row"\,j$ cross at a unique $0$ of type $EE$.

In the above, we showed that each $(i, j)\in  A_{nest}(\sigma)$ corresponds to a unique $0$ either of type $EE$ or of type $NN$. It can be easily shown that two zigzag paths passing through a $0$ of type $EE$ or $NN$ are from $\verb"row"\,i$ and $\verb"row"\,j$ for $i,j\in[\pm n]$ or $\verb"col"\,k$ and $\verb"col"\,l$ for $k,l\in[n]$, respectively. This proves the first part of the proposition.

To prove $\al_{EN}(\zeta(\T))=\zero_{EN}(\T)$, recall from Definition~\ref{def_al} that a pair $(i, j)\in [n]\times[n]$ is in $A_{EN}(\sigma)$ if and only if $-i<0<-\sigma(i)<\sigma(j)<j$ or $i\le\sigma(i)<\sigma(j)<j$.

Suppose that $-i<0<-\sigma(i)<\sigma(j)<j$ holds for a pair $(i, j)\in A_{EN}(\sigma)$. Then two zigzag paths from $\verb"row"\,(-i)$ and $\verb"col"\,j$ cross at a unique $0$ of type $EN$; see first two figures in Figure~\ref{pathen1}. Suppose that $i\le\sigma(i)<\sigma(j)<j$ holds for a pair $(i,j)\in A_{EN}(\sigma)$. Then two zigzag paths from $\verb"row"\,i$ and $\verb"col"\,j$ cross at a unique zero of type $EN$. See the third figure in Figure~\ref{pathen1}. Therefore, each $(i, j)\in  A_{EN}(\sigma)$ corresponds to a unique $0$ of type $EN$ and it is easy to show that two zigzag paths passing through a $0$ of type $EN$ are from $\verb"row"\,i$ and $\verb"col"\,j$ for $i\in[\pm n]$, $j\in[n]$, which completes the proof of the second argument.

  \begin{figure}[!ht]
    \begin{tikzpicture}\tiny
      \draw[-] (1.0,8.0) to (1.0,4.1);
      \draw[-] (1.0,8.0) to (1.3,8.0);
      \draw[-] (1.3,8.0) to (1.3,7.7);
      \draw[-] (1.3,7.7) to (1.6,7.7);
      \draw[-] (1.6,7.7) to (1.6,7.4);
      \draw[-] (1.6,7.4) to (1.9,7.4);
      \draw[-] (1.9,7.4) to (1.9,7.1);
      \draw[-] (1.9,7.1) to (2.2,7.1);
      \draw[-] (2.2,7.1) to (2.2,6.8);
      \draw[-] (2.2,6.8) to (2.5,6.8);
      \draw[-] (2.5,6.8) to (2.5,6.5);
      \draw[-] (2.5,6.5) to (2.8,6.5);
      \draw[-] (2.8,6.5) to (2.8,6.2);
      \draw[-] (2.8,6.2) to (3.1,6.2);
      \draw[-] (3.1,6.2) to (3.1,5.0);
      \draw[-] (3.1,5.0) to (2.8,5.0);
      \draw[-] (2.8,5.0) to (2.8,4.4);
      \draw[-] (2.8,4.4) to (2.2,4.4);
      \draw[-] (2.2,4.4) to (2.2,4.1);
      \draw[-] (2.2,4.1) to (1.0,4.1);
      \node at (0.7,7.25) {$-i$};
      \tikzstyle{Element}=[draw,minimum width=3mm, minimum height=3mm, node distance=3mm, inner sep=0pt]
      \node [Element] at (1.75,7.25) {1};
      \node [Element] [label=above:$j$] at (2.35,6.65) {0};
      \node [Element,fill=lightgray] at (2.35,6.0) {\texttt{EN}};
      \path[-] (1.0,7.25) edge (1.75,7.25);
      \path[-] (1.75,7.25) edge (1.75,6.0);
      \path[-] (1.75,6.0) edge (2.65,6.0);
      \path[-] (2.65,6.0) edge (2.65,5.3);
      \path[->] (2.65,5.3) edge (3.1,5.3);
      \node at (3.55,5.3) {$\sigma(-i)$};
      \path[dashed] (2.35,6.8) edge (2.35,5.1);
      \path[dashed] (2.35,5.1) edge (2.65,5.1);
      \path[dashed] (2.65,5.1) edge (2.65,4.6);
      \path[dashed,->] (2.65,4.6) edge (2.8,4.6);
      \node at (3.15,4.6) {$\sigma(j)$};
      \normalsize\node at (2.0,3.7) {$i>j$};
    \tiny
      \draw[-] (5.0,8.0) to (5.0,4.1);
      \draw[-] (5.0,8.0) to (5.3,8.0);
      \draw[-] (5.3,8.0) to (5.3,7.7);
      \draw[-] (5.3,7.7) to (5.6,7.7);
      \draw[-] (5.6,7.7) to (5.6,7.4);
      \draw[-] (5.6,7.4) to (5.9,7.4);
      \draw[-] (5.9,7.4) to (5.9,7.1);
      \draw[-] (5.9,7.1) to (6.2,7.1);
      \draw[-] (6.2,7.1) to (6.2,6.8);
      \draw[-] (6.2,6.8) to (6.5,6.8);
      \draw[-] (6.5,6.8) to (6.5,6.5);
      \draw[-] (6.5,6.5) to (6.8,6.5);
      \draw[-] (6.8,6.5) to (6.8,6.2);
      \draw[-] (6.8,6.2) to (7.1,6.2);
      \draw[-] (7.1,6.2) to (7.1,5.0);
      \draw[-] (7.1,5.0) to (6.8,5.0);
      \draw[-] (6.8,5.0) to (6.8,4.4);
      \draw[-] (6.8,4.4) to (6.2,4.4);
      \draw[-] (6.2,4.4) to (6.2,4.1);
      \draw[-] (6.2,4.1) to (5.0,4.1);
      \node at (4.7,6.65) {$-i$};
      \tikzstyle{Element}=[draw,minimum width=3mm, minimum height=3mm, node distance=3mm, inner sep=0pt]
      \node [Element] at (6.35,6.65) {1};
      \node [Element] [label=above:$j$] at (5.75,7.25) {0};
      \node [Element,fill=lightgray] at (5.75,5.9) {\texttt{EN}};
      \path[-] (5.0,6.65) edge (5.45,6.65);
      \path[-] (5.45,6.65) edge (5.45,5.9);
      \path[-] (5.45,5.9) edge (6.35,5.9);
      \path[-] (6.35,5.9) edge (6.35,5.3);
      \path[->] (6.35,5.3) edge (7.1,5.3);
      \node at (7.55,5.3) {$\sigma(-i)$};
      \path[dashed] (5.75,7.4) edge (5.75,5.1);
      \path[dashed] (5.75,5.1) edge (6.35,5.1);
      \path[dashed,->] (6.35,5.1) edge (6.35,4.4);
      \node at (6.55,4.2) {$\sigma(j)$};
      \normalsize\node at (6.0,3.7) {$i<j$};
      \node at (4.0,3.2) {$-i<0<-\sigma(i)<\sigma(j)<j$};
    \end{tikzpicture}\normalsize\quad\quad\quad\quad
    \begin{tikzpicture}\tiny
      \draw[-] (1.0,8.0) to (1.0,4.1);
      \draw[-] (1.0,8.0) to (1.3,8.0);
      \draw[-] (1.3,8.0) to (1.3,7.7);
      \draw[-] (1.3,7.7) to (1.6,7.7);
      \draw[-] (1.6,7.7) to (1.6,7.4);
      \draw[-] (1.6,7.4) to (1.9,7.4);
      \draw[-] (1.9,7.4) to (1.9,7.1);
      \draw[-] (1.9,7.1) to (2.2,7.1);
      \draw[-] (2.2,7.1) to (2.2,6.8);
      \draw[-] (2.2,6.8) to (2.5,6.8);
      \draw[-] (2.5,6.8) to (2.5,6.5);
      \draw[-] (2.5,6.5) to (2.8,6.5);
      \draw[-] (2.8,6.5) to (2.8,6.2);
      \draw[-] (2.8,6.2) to (3.1,6.2);
      \draw[-] (3.1,6.2) to (3.1,5.0);
      \draw[-] (3.1,5.0) to (2.8,5.0);
      \draw[-] (2.8,5.0) to (2.8,4.4);
      \draw[-] (2.8,4.4) to (2.2,4.4);
      \draw[-] (2.2,4.4) to (2.2,4.1);
      \draw[-] (2.2,4.1) to (1.0,4.1);
      \node at (0.8,5.75) {$i$};
      \tikzstyle{Element}=[draw,minimum width=3mm, minimum height=3mm, node distance=3mm, inner sep=0pt]
      \node [Element] [label=above:$j$] at (1.75,7.25) {0};
      \node [Element,fill=lightgray] at (2.35,5.75) {\texttt{EN}};
      \path[-] (1.0,5.75) edge (2.65,5.75);
      \path[-] (2.65,5.75) edge (2.65,5.3);
      \path[->] (2.65,5.3) edge (3.1,5.3);
      \node at (3.45,5.3) {$\sigma(i)$};
      \path[dashed] (1.75,7.4) edge (1.75,6.35);
      \path[dashed] (1.75,6.35) edge (2.35,6.35);
      \path[dashed,->] (2.35,6.35) edge (2.35,4.4);
      \node at (2.55,4.2) {$\sigma(j)$};
      \normalsize
      \node at (2.0,3.2) {$i\le\sigma(i)<\sigma(j)<j$};
    \end{tikzpicture}\normalsize\vspace{-3mm}
    \caption{Two zigzag paths}
    \label{pathen1}
  \end{figure}

Recall from the definition that $(i,j)\in [n]\times [n] $ is in $A_{NE}(\sigma)$ if and only if $\sigma(i)<i<j\le\sigma(j)$. Hence, if $(i,j)\in A_{NE}(\sigma)$ then $i\in\drop(\sigma)$, $j\in\wex(\sigma)$ and $i<j$. In this case, we let $(i,j)\in A_{NE}(\sigma)$ correspond to the box in $\verb"row"\,j$ and $\verb"col"\,i$ of the extended representation of $\T$, which is filled with $2$. It is easy to check that this correspondence is bijective.
\end{proof}

We prove a theorem for the sum of alignments and crossings of a signed permutation, which extends the result on permutations by Corteel in \cite{SC}.

\begin{theorem}\label{sumofalcr}
  Let $\sigma\in\sym_{n}^{B}$. Then,
    $$\al(\sigma)+\crs(\sigma)=(n- \frac{1}{2}\fwex(\sigma))(\wex(\sigma)-1+\n(\sigma))+\frac{1}{2}\n(\sigma)\wex(\sigma).$$
\end{theorem}
\begin{proof} Let $\T=\zeta^{-1}(\sigma)\in\mathcal{PT}_{n}^{B}$. Then, by Theorem~\ref{alignments} and Proposition~\ref{propdiagram}, we have
$\al(\sigma)=\al_{nest}(\sigma)+\al_{EN}(\sigma)+\al_{NE}(\sigma)=\zero(\T)+\two(\T)$  and $\crs(\sigma)=\so(\T)$. Therefore,
$$\al(\sigma)+\crs(\sigma)=\zero(\T)+\two(\T)+\so(\T)=\zero(\T)+\two(\T)+\one(\T)-\col(\mathcal{T})\,,$$
where $\col(\mathcal{T})=n-\row(\mathcal{T})$ is the number of columns, hence the number of topmost 1's, in $\T$. In the extended representation of $\T$, there are $n+(n-1)+\cdots+(n-\col(\mathcal{T})+1)=\frac{1}{2}\col(\mathcal{T})(2n-\col(\mathcal{T})+1)$ boxes and this is equal to  $\zero(\T)+\two(\T)+\one(\T)$ plus the number of non-typed zeros in $\T$. Moreover, there are $1+\cdots+(\col(\T)-\diag(\T))=\frac{1}{2} (\col(\T)-\diag(\T)) (1+\col(\T)-\diag(\T))$ non-typed $0$'s by Lemma~\ref{non-typed zero}, since the number of zero rows with negative labels is $(\col(\T)-\diag(\T))$.
Therefore, $\al(\sigma)+\crs(\sigma)$ is $$\frac{1}{2}\col(\mathcal{T})(2n-\col(\mathcal{T})+1)-\frac{1}{2}(\col(\mathcal{T})-\diag(\mathcal{T}))(\col(\mathcal{T})-\diag(\mathcal{T})+1)-\col(\T),$$
and using the relations $\row(\mathcal{T})=\wex(\sigma)$, $\diag(\T)=\n(\sigma)$, and $2\wex(\sigma)+\n(\sigma)=\fwex(\sigma)$, we finally have
  $$\al(\sigma)+\crs(\sigma)=(n- \frac{1}{2}\fwex(\sigma))(\wex(\sigma)-1+\n(\sigma))+\frac{1}{2}\n(\sigma)\wex(\sigma).$$
\end{proof}

\begin{remark}
  For a permutation $\sigma\in\sym_{n}$, since $\n(\sigma)=0$, the equation in Theorem~\ref{sumofalcr} becomes $\al(\sigma)+\crs(\sigma)=(n-\wex(\sigma))(\wex(\sigma)-1)$ which is Proposition 5 in \cite{ES-LW}. A similar equation in Theorem~\ref{sumofalcr} is proved by Corteel et al. (Proposition 4.3 in \cite{SC-MJV-JK}) using pignose diagram, and it is posed as an open problem to give a proof of the proposition using permutation tableaux of type $B$. Our proof of Theorem~\ref{sumofalcr} refines the equation and resolves the problem of Corteel et al. in \cite{SC-MJV-JK}; remember that the definition of `alignments' of ours is different from the one given by Corteel et al. in \cite{SC-MJV-JK}.
\end{remark}

\begin{example}\label{exsumofalcr}
  Let $\T=\zeta^{-1}(\sigma)$ for $\sigma=-2,-4,5,3,1\in\sym_{5}^{B}$. We know from Example~\ref{ex_alcr} that $\al_{nest}(\sigma)=2$, $\al_{EN}(\sigma)=1$, $\al_{NE}(\sigma)=2$, and $\crs(\sigma)=2$. The extended representation of $\T$ is in Figure~\ref{ptex} and we have $\zero_{EE}(\T)=1$, $\zero_{NN}(\T)=1$, $\zero_{EN}(\T)=1$, and $\two(\T)=2$. Hence, $\al_{nest}(\sigma)=\zero_{EE}(\T)+\zero_{NN}(\T)=2$, $\al_{EN}(\sigma)=\zero_{EN}(\T)=1$, and $\al_{NE}(\sigma)=\two(\T)=2$. Moreover, since $\wex(\sigma)=1$, $\n(\sigma)=2$, and $\fwex(\sigma)=4$, we can check that $\al(\sigma)+\crs(\sigma)=(5-\frac{1}{2}\times 4)(1-1+2)+\frac{1}{2}\times 2 \times 1 = 7.$
  \begin{figure}[!ht]
    \begin{tikzpicture}
    \tiny
    \tikzstyle{Element} = [draw, minimum width=4mm, minimum height=4mm, inner sep=0pt]
      \node [Element] [label=left:-5] at (1.0,6.0) {0};
      \node [Element] [label=left:-4] at (1.0,5.6) {0};
      \node [Element] [label=left:-2] at (1.0,5.2) {1};
      \node [Element] [label=left:-1] at (1.0,4.8) {\texttt{EE}};
      \node [Element] [label=left:3] at (1.0,4.4) {1};

      \node [Element] at (1.4,5.6) {0};
      \node [Element] at (1.4,5.2) {\texttt{NN}};
      \node [Element] at (1.4,4.8) {\texttt{EN}};
      \node [Element] at (1.4,4.4) {1};

      \node [Element] at (1.8,5.2) {1};
      \node [Element] at (1.8,4.8) {1};
      \node [Element] at (1.8,4.4) {2};

      \node [Element] at (2.2,4.8) {1};
      \node [Element] at (2.2,4.4) {2};
    \end{tikzpicture}\normalsize
    \vspace{-3mm}
  \caption{The extended representation of $\T=\zeta^{-1}(\sigma)$ where $\sigma=-2,-4,5,3,1$}\label{ptex}
  \end{figure}
\end{example} 

%% file: bruhat.tex
\section{Weak Bruhat order on Coxeter system of type $B$}\label{bruhat}

In this section, we consider the  weak Bruhat order on Coxeter system of type $B$ in terms of  permutation tableaux of type $B$.

The group $\sym_{n}^{B}$ of signed permutations with generators $s_{0},s_{1},\ldots,s_{n-1}$ forms the Coxeter system of type $B$, where $s_{0}=(-1)$, $s_{i}=(i,i+1)$ for $i\in[n-1]$ in cycle notation. The \emph{length} of $\sigma\in\sym_{n}^{B}$, denoted by $\ell(\sigma)$, is defined as the minimum of $k$ such that $\sigma=s_{a_{1}}\ldots s_{a_{k}}$, and an expression $\sigma=s_{a_{1}}\ldots s_{a_{k}}$ is \emph{reduced} if $k=\ell(\sigma)$. A pair $(i,j)\in[n]\times[n]$ is called an \emph{inversion} of $\sigma$, if $i<j$ and $\sigma(i)>\sigma(j)$, or $i\le j$ and $\sigma(-i)>\sigma(j)$. The number of inversions of $\sigma$ is denoted by $\inv(\sigma)$ and it is well known that $\ell(\sigma)=\inv(\sigma)$. (See Proposition 8.1.1 in \cite{AB-FB}.)

\begin{example}\label{exinv}
  Let $\sigma =-2,-4,5,3,1\in\sym_{5}^{B}$ in Example~\ref{ex_alcr}.
  Then, $\{(i,j)\in[n]\times[n]\,|\, i<j\mbox{ and } \sigma(i)>\sigma(j)\} = \{(1,2),(3,4),(3,5),(4,5)\}$ and $\{(i,j)\in[n]\times[n]\,|\, i\le j\mbox{ and } \sigma(-i)>\sigma(j)\} = \{(1,1),(1,2),(1,5),(2,2),(2,4),(2,5)\}$, hence, $\ell(\sigma)=\inv(\sigma)=10$.
\end{example}

The \emph{covering relation in weak Bruhat order} on $\sym_{n}^{B}$ is defined as follows; for $\sigma,\sigma'\in\sym_{n}^{B}$, $\sigma'$ \emph{covers} $\sigma$, denoted by $\sigma\lhd\sigma'$,  if $\sigma'=\sigma s_{i}$ for some $i\in\{0,1,\ldots,n-1\}$ and $\ell(\sigma')=\ell(\sigma)+1$.

The following lemma is useful to understand the weak Bruhat order in permutation tableaux of type $B$.

\begin{lemma}\label{rule} Let $i\in [n-1]$. For $*=1, 2$,
let $\T_{*}$ and $\T_{*}'$ be permutation tableaux of type $B$, where $\T_{*}'$ is obtained  by replacing a rectangular part of $\T_{*}$ with another as in Figure~\ref{rule1} and Figure~\ref{rule2}. Then, $\sigma_{*}'=\sigma_{*} s_{i}$, where $\sigma_{*}=\zeta(\T_{*})$ and $\sigma_{*}'=\zeta(\T_{*}')$.
Moreover, the type of $0_{j'}$ in $\T_*'$ is the same as the type of $0_{j}$ in $\T_*$ for all $j$.
\end{lemma}
\begin{proof}
One can check that  $\sigma_{*}(i)=\sigma_{*}'(i+1)$ and $\sigma_{*}(i+1)=\sigma_{*}'(i)$ while $\sigma_{*}(j)=\sigma_{*}'(j)$ for all $j\ne i, i+1$ in Figure~\ref{rule1} and Figure~\ref{rule2}. That is, $\sigma_{*}'=\sigma_{*} s_{i}$ for $*=1, 2$. Moreover, it is easy to check that the zigzag path passing through $0_{j'}$ vertically (or horizontally) in $\T_*'$ is from the same row or column with the one passing through $0_j$ vertically (or horizontally, respectively) in $\T_*$.
\end{proof}
 \begin{figure}[!ht]
   \begin{tikzpicture}
     \tikzstyle{Element}=[draw, minimum width=5mm, minimum height=4mm, inner sep=0pt]
       \tiny
       \node[Element] at (0.0,4.0) {$0_{1}$};
       \node[Element] at (0.5,4.0) {$\cdots$};
       \node[Element] at (1.0,4.0) {}; \node at (1.02,4.0) {$\cdots$};
       \node[Element] at (1.5,4.0) {$0_{k}$};
       \node[Element,fill=lightgray] at (2.0,4.0) {1};
       \node[Element,fill=lightgray] at (0.0,3.6) {1};
       \node[Element] at (0.5,3.6) {1};
       \node[Element] at (1.0,3.6) {}; \node at (1.02,3.6) {$\cdots$};
       \node[Element] at (1.5,3.6) {1};
       \node[Element] at (2.0,3.6) {1};
       \node (i1) at (-0.6,4) {$i$}; \node (i2) at (-0.9,3.6) {$i+1$};
       \draw[-,gray] (i1) to (2.0,4.0); \draw[-,gray] (2.0,4) to (2.0,3.6); \draw[->,gray] (2.0,3.6) to (2.5,3.6);
       \draw[-,gray] (i2) to (0,3.6); \draw[->,gray] (0,3.6) to (0,3);
       \node[Element,fill=lightgray] at (-0.1,4.7) {1};
       \node at (1,4.7) {: leftmost 1};

       \node[Element] at (4,4.0) {1};
       \node[Element] at (4.5,4.0) {1};
       \node[Element] at (5,4.0) {}; \node at (5.02,4.0) {$\cdots$};
       \node[Element] at (5.5,4.0) {1};
       \node[Element] at (6,4.0) {1};
       \node[Element] at (4,3.6) {\texttt{EE}};
       \node[Element] at (4.5,3.6) {$0_{1'}$};
       \node[Element] at (5,3.6) {}; \node at (5.02,3.6) {$\cdots$};
       \node[Element] at (5.5,3.6) {$\cdots$};
       \node[Element] at (6,3.6) {$0_{k'}$};
       \node (i3) at (3.5,4) {$i$}; \node (i4) at (3.2,3.6) {$i+1$};
       \draw[-,gray] (i3) to (4,4); \draw[->,gray] (4,4) to (4,3);
       \draw[->,dashed,gray] (i4) to (6.5,3.6);
       \small \node at (1.0,2.5) {a part of $\T_{1}$}; \node at (5,2.5) {a part of $\T_{1}'$};
   \end{tikzpicture}\normalsize\quad\quad\quad\quad\quad
   \begin{tikzpicture}
     \tikzstyle{Element}=[draw, minimum width=5mm, minimum height=4mm, inner sep=0pt]
       \tiny
       \node[Element] at (0.0,4.0) {$0_{1}$};
       \node[Element] at (0.5,4.0) {$\cdots$};
       \node[Element] at (1.0,4.0) {}; \node at (1.02,4.0) {$\cdots$};
       \node[Element] at (1.5,4.0) {$0_{k}$};
       \node[Element] at (2.0,4.0) {1};
       \node[Element] at (0.0,3.6) {1};
       \node[Element] at (0.5,3.6) {1};
       \node[Element] at (1.0,3.6) {}; \node at (1.02,3.6) {$\cdots$};
       \node[Element] at (1.5,3.6) {1};
       \node[Element] at (2.0,3.6) {1};
       \node (j1) at (0,4.6) {$j_1$}; \node (j2) at (1.5,4.6) {$j_2$};
       \draw[-,gray] (j1) to (0,3.6); \draw[-,gray] (0,3.6) to (0.5,3.6); \draw[->,gray] (0.5,3.6) to (0.5,3);
       \draw[-,gray] (j2) to (1.5,3.6); \draw[-,gray] (1.5,3.6) to (2,3.6); \draw[->,gray] (2,3.6) to (2,3);

       \node[Element] at (3,4.0) {1};
       \node[Element] at (3.5,4.0) {1};
       \node[Element] at (4,4.0) {}; \node at (4.02,4.0) {$\cdots$};
       \node[Element] at (4.5,4.0) {1};
       \node[Element] at (5,4.0) {1};
       \node[Element] at (3,3.6) {\texttt{EE}};
       \node[Element] at (3.5,3.6) {$0_{1'}$};
       \node[Element] at (4,3.6) {}; \node at (4.02,3.6) {$\cdots$};
       \node[Element] at (4.5,3.6) {$\cdots$};
       \node[Element] at (5,3.6) {$0_{k'}$};
       \node (j3) at (3,4.6) {$j_1$}; \node (j4) at (4.5,4.6) {$j_2$};
       \draw[-,gray] (j3) to (3,4); \draw[-,gray] (3,4) to (3.5,4); \draw[->,gray] (3.5,4) to (3.5,3);
       \draw[-,gray] (j4) to (4.5,4); \draw[-,gray] (4.5,4) to (5,4); \draw[->,gray] (5,4) to (5,3);
       \small \node at (1,2.5) {a part of $\T_{1}$}; \node at (4,2.5) {a part of $\T_{1}'$};
   \end{tikzpicture}\normalsize\vspace{-3mm}
   \caption{Zigzag paths in $\T_1$ and $\T'_1$}\label{rule1}
 \end{figure}

 \begin{figure}[!ht]
   \begin{tikzpicture}
     \tikzstyle{Element}=[draw, minimum width=5mm, minimum height=4mm, inner sep=0pt]
       \tiny
       \node[Element] at (0,4) {$0_{1}$};
       \node[Element] at (0,3.6) {}; \node at (0,3.7) {$\cdot$};\node at (0,3.6) {$\cdot$};\node at (0,3.5) {$\cdot$};
       \node[Element] at (0,3.2) {}; \node at (0,3.3) {$\cdot$};\node at (0,3.2) {$\cdot$};\node at (0,3.1) {$\cdot$};
       \node[Element] at (0,2.8) {$0_{k}$};
       \node[Element,fill=lightgray] at (0,2.4) {1};
       \node[Element,fill=lightgray] at (0.5,4) {1};
       \node[Element] at (0.5,3.6) {1};
       \node[Element] at (0.5,3.2) {}; \node at (0.5,3.3) {$\cdot$}; \node at (0.5,3.2) {$\cdot$};\node at (0.5,3.1) {$\cdot$};
       \node[Element] at (0.5,2.8) {1};
       \node[Element] at (0.5,2.4) {1};
       \node (i1) at (0,4.5) {$i\hspace{-1mm} +\hspace{-1mm}1$}; \node (i2) at (0.5,4.5) {$i$};
       \draw[-,gray] (i1) to (0,2.4); \draw[-,gray] (0,2.4) to (0.5,2.4); \draw[->,gray] (0.5,2.4) to (0.5,1.8);
       \draw[-,gray] (i2) to (0.5,4); \draw[->,gray] (0.5,4) to (1,4);
       \node[Element,fill=lightgray] at (1.3,3) {1};
       \node at (2.4,3) {: topmost 1};
       \small \node at (0.3,1.4) {a part of $\T_{2}$};
   \end{tikzpicture}\normalsize
   \begin{tikzpicture}
     \tikzstyle{Element}=[draw, minimum width=5mm, minimum height=4mm, inner sep=0pt]
       \tiny
       \node[Element] at (0,4) {1};
       \node[Element] at (0,3.6) {1};
       \node[Element] at (0,3.2) {}; \node at (0,3.3) {$\cdot$};\node at (0,3.2) {$\cdot$};\node at (0,3.1) {$\cdot$};
       \node[Element] at (0,2.8) {1};
       \node[Element] at (0,2.4) {1};
       \node[Element] at (0.5,4) {\texttt{NN}};
       \node[Element] at (0.5,3.6) {$0_{1'}$};
       \node[Element] at (0.5,3.2) {}; \node at (0.5,3.3) {$\cdot$}; \node at (0.5,3.2) {$\cdot$};\node at (0.5,3.1) {$\cdot$};
       \node[Element] at (0.5,2.8) {}; \node at (0.5,2.9) {$\cdot$}; \node at (0.5,2.8) {$\cdot$};\node at (0.5,2.7) {$\cdot$};
       \node[Element] at (0.5,2.4) {$0_{k'}$};
       \node (i1) at (0,4.5) {$i\hspace{-1mm} +\hspace{-1mm}1$}; \node (i2) at (0.5,4.5) {$i$};
       \draw[-,gray] (i1) to (0,4); \draw[->,gray] (0,4) to (1,4);
       \draw[->,dashed,gray] (i2) to (0.5,1.8);
       \small \node at (0.3,1.4) {a part of $\T_{2}'$};
   \end{tikzpicture}\normalsize\quad\quad\quad\quad\quad\quad
   \begin{tikzpicture}
     \tikzstyle{Element}=[draw, minimum width=5mm, minimum height=4mm, inner sep=0pt]
       \tiny
       \node[Element] at (0,4) {$0_{1}$};
       \node[Element] at (0,3.6) {}; \node at (0,3.7) {$\cdot$};\node at (0,3.6) {$\cdot$};\node at (0,3.5) {$\cdot$};
       \node[Element] at (0,3.2) {}; \node at (0,3.3) {$\cdot$};\node at (0,3.2) {$\cdot$};\node at (0,3.1) {$\cdot$};
       \node[Element] at (0,2.8) {$0_{k}$};
       \node[Element] at (0,2.4) {1};
       \node[Element] at (0.5,4) {1};
       \node[Element] at (0.5,3.6) {1};
       \node[Element] at (0.5,3.2) {}; \node at (0.5,3.3) {$\cdot$}; \node at (0.5,3.2) {$\cdot$};\node at (0.5,3.1) {$\cdot$};
       \node[Element] at (0.5,2.8) {1};
       \node[Element] at (0.5,2.4) {1};
       \node (j1) at (-0.8,4) {$j_{1}$}; \node (j2) at (-0.8,2.8) {$j_{2}$};
       \draw[-,gray] (j1) to (0.5,4); \draw[-,gray] (0.5,4) to (0.5,3.6); \draw[->,gray] (0.5,3.6) to (1,3.6);
       \draw[-,gray] (j2) to (0.5,2.8); \draw[-,gray] (0.5,2.8) to (0.5,2.4); \draw[->,gray] (0.5,2.4) to (1,2.4);
       \small \node at (0.3,1.4) {a part of $\T_{2}$};
   \end{tikzpicture}\normalsize\quad\quad
   \begin{tikzpicture}
     \tikzstyle{Element}=[draw, minimum width=5mm, minimum height=4mm, inner sep=0pt]
       \tiny
       \node[Element] at (0,4) {1};
       \node[Element] at (0,3.6) {1};
       \node[Element] at (0,3.2) {}; \node at (0,3.3) {$\cdot$};\node at (0,3.2) {$\cdot$};\node at (0,3.1) {$\cdot$};
       \node[Element] at (0,2.8) {1};
       \node[Element] at (0,2.4) {1};
       \node[Element] at (0.5,4) {\texttt{NN}};
       \node[Element] at (0.5,3.6) {$0_{1'}$};
       \node[Element] at (0.5,3.2) {}; \node at (0.5,3.3) {$\cdot$}; \node at (0.5,3.2) {$\cdot$};\node at (0.5,3.1) {$\cdot$};
       \node[Element] at (0.5,2.8) {}; \node at (0.5,2.9) {$\cdot$}; \node at (0.5,2.8) {$\cdot$};\node at (0.5,2.7) {$\cdot$};
       \node[Element] at (0.5,2.4) {$0_{k'}$};
       \node (j1) at (-0.8,4) {$j_{1}$}; \node (j2) at (-0.8,2.8) {$j_{2}$};
       \draw[-,gray] (j1) to (0,4); \draw[-,gray] (0,4) to (0,3.6); \draw[->,gray] (0,3.6) to (1,3.6);
       \draw[-,gray] (j2) to (0,2.8); \draw[-,gray] (0,2.8) to (0,2.4); \draw[->,gray] (0,2.4) to (1,2.4);
       \small \node at (0.3,1.4) {a part of $\T_{2}'$};
   \end{tikzpicture}\normalsize\vspace{-3mm}
   \caption{Zigzag paths in $\T_2$ and $\T'_2$}\label{rule2}
 \end{figure}

\begin{remark} We can show that Lemma~\ref{rule} is also true if we consider two consecutive `negative' rows $\verb"row"\,(-i)$, $\verb"row"\,(-(i+1))$, $i\in [n]$ in Figure~\ref{rule1} instead of positive rows.
\end{remark}

Our main goal is to illustrate the covering relations in weak Bruhat order on $\sym_{n}^{B}$ in terms of permutation tableaux of type $B$: We describe how $\T=\zeta^{-1}(\sigma)$ and $\T'=\zeta^{-1}(\sigma')$ are related when $\sigma\lhd\sigma'$.
Suppose that $\sigma\lhd\sigma'$ in $\sym_{n}^{B}$ and let $\T=\zeta^{-1}(\sigma)$ and $\T'=\zeta^{-1}(\sigma')$. Then $\sigma'=\sigma\, s_{i}$ for some $i\in\{0,1,\ldots,n-1\}$ and $\ell(\sigma')=\ell(\sigma)+1$.

(WB1): If $s_i=s_0$, then $\sigma(1)$ must be positive, hence $\sigma'(1)$ is negative, and $\T'$ is easily obtained by adding one diagonal with 1; see Figure \ref{weakbruhat_p1ton1} and check that $\sigma'(1)=-\sigma(1)$.
 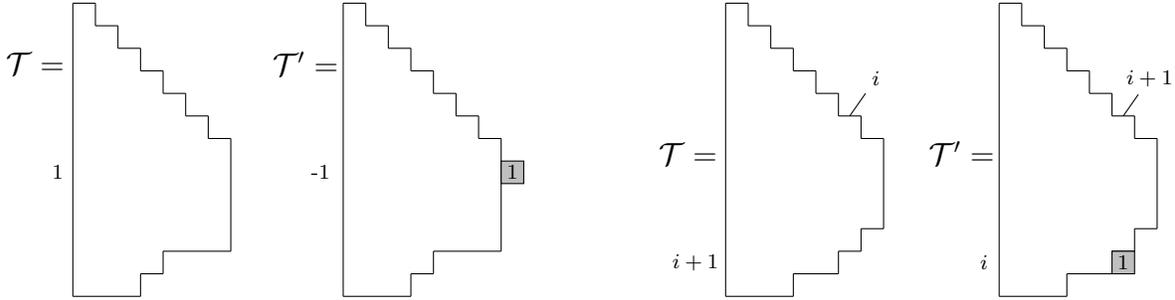
\begin{figure}[!ht]
    \begin{tikzpicture}
            \node at (0.5,7.2) {$\T=$};\tiny
            \draw[-] (1.0,8.0) to (1.0,4.1);
            \draw[-] (1.0,8.0) to (1.3,8.0);
            \draw[-] (1.3,8.0) to (1.3,7.7);
            \draw[-] (1.3,7.7) to (1.6,7.7);
            \draw[-] (1.6,7.7) to (1.6,7.4);
            \draw[-] (1.6,7.4) to (1.9,7.4);
            \draw[-] (1.9,7.4) to (1.9,7.1);
            \draw[-] (1.9,7.1) to (2.2,7.1);
            \draw[-] (2.2,7.1) to (2.2,6.8);
            \draw[-] (2.2,6.8) to (2.5,6.8);
            \draw[-] (2.5,6.8) to (2.5,6.5);
            \draw[-] (2.5,6.5) to (2.8,6.5);
            \draw[-] (2.8,6.5) to (2.8,6.2);
            \draw[-] (2.8,6.2) to (3.1,6.2);
            \draw[-] (3.1,6.2) to (3.1,4.7);
            \draw[-] (3.1,4.7) to (2.2,4.7);
            \draw[-] (2.2,4.7) to (2.2,4.4);
            \draw[-] (2.2,4.4) to (1.9,4.4);
            \draw[-] (1.9,4.4) to (1.9,4.1);
            \draw[-] (1.9,4.1) to (1.0,4.1);
            \node at (0.8,5.75) {1};
          \end{tikzpicture}\normalsize\quad
    \begin{tikzpicture}
      \node at (0.5,7.2) {$\T'=$};\tiny
      \draw[-] (1.0,8.0) to (1.0,4.1);
      \draw[-] (1.0,8.0) to (1.3,8.0);
      \draw[-] (1.3,8.0) to (1.3,7.7);
      \draw[-] (1.3,7.7) to (1.6,7.7);
      \draw[-] (1.6,7.7) to (1.6,7.4);
      \draw[-] (1.6,7.4) to (1.9,7.4);
      \draw[-] (1.9,7.4) to (1.9,7.1);
      \draw[-] (1.9,7.1) to (2.2,7.1);
      \draw[-] (2.2,7.1) to (2.2,6.8);
      \draw[-] (2.2,6.8) to (2.5,6.8);
      \draw[-] (2.5,6.8) to (2.5,6.5);
      \draw[-] (2.5,6.5) to (2.8,6.5);
      \draw[-] (2.8,6.5) to (2.8,6.2);
      \draw[-] (2.8,6.2) to (3.1,6.2);
      \draw[-] (3.1,6.2) to (3.1,4.7);
      \draw[-] (3.1,4.7) to (2.2,4.7);
      \draw[-] (2.2,4.7) to (2.2,4.4);
      \draw[-] (2.2,4.4) to (1.9,4.4);
      \draw[-] (1.9,4.4) to (1.9,4.1);
      \draw[-] (1.9,4.1) to (1.0,4.1);
      \node at (0.7,5.75) {-1};
      \tikzstyle{Element}=[draw,minimum width=3mm, minimum height=3mm, inner sep=0pt]
      \node[Element,fill=lightgray] at (3.25,5.75) {1};
    \end{tikzpicture}\normalsize\quad\quad\quad\quad
    \begin{tikzpicture}
      \node at (0.5,6.0) {$\T=$};\tiny
      \draw[-] (1.0,8.0) to (1.0,4.1);
      \draw[-] (1.0,8.0) to (1.3,8.0);
      \draw[-] (1.3,8.0) to (1.3,7.7);
      \draw[-] (1.3,7.7) to (1.6,7.7);
      \draw[-] (1.6,7.7) to (1.6,7.4);
      \draw[-] (1.6,7.4) to (1.9,7.4);
      \draw[-] (1.9,7.4) to (1.9,7.1);
      \draw[-] (1.9,7.1) to (2.2,7.1);
      \draw[-] (2.2,7.1) to (2.2,6.8);
      \draw[-] (2.2,6.8) to (2.5,6.8);
      \draw[-] (2.5,6.8) to (2.5,6.5);
      \draw[-] (2.5,6.5) to (2.8,6.5);
      \draw[-] (2.8,6.5) to (2.8,6.2);
      \draw[-] (2.8,6.2) to (3.1,6.2);
      \draw[-] (3.1,6.2) to (3.1,5.0);
      \draw[-] (3.1,5.0) to (2.8,5.0);
      \draw[-] (2.8,5.0) to (2.8,4.7);
      \draw[-] (2.8,4.7) to (2.5,4.7);
      \draw[-] (2.5,4.7) to (2.5,4.4);
      \draw[-] (2.5,4.4) to (1.9,4.4);
      \draw[-] (1.9,4.4) to (1.9,4.1);
      \draw[-] (1.9,4.1) to (1.0,4.1);
      \node at (0.6,4.55) {$i+1$};
      \node (i1) at (3,7) {$i$};
      \draw[-] (2.65,6.5) to (i1);
    \end{tikzpicture}\normalsize\quad
    \begin{tikzpicture}
      \node at (0.5,6.0) {$\T'=$};\tiny
      \draw[-] (1.0,8.0) to (1.0,4.1);
      \draw[-] (1.0,8.0) to (1.3,8.0);
      \draw[-] (1.3,8.0) to (1.3,7.7);
      \draw[-] (1.3,7.7) to (1.6,7.7);
      \draw[-] (1.6,7.7) to (1.6,7.4);
      \draw[-] (1.6,7.4) to (1.9,7.4);
      \draw[-] (1.9,7.4) to (1.9,7.1);
      \draw[-] (1.9,7.1) to (2.2,7.1);
      \draw[-] (2.2,7.1) to (2.2,6.8);
      \draw[-] (2.2,6.8) to (2.5,6.8);
      \draw[-] (2.5,6.8) to (2.5,6.5);
      \draw[-] (2.5,6.5) to (2.8,6.5);
      \draw[-] (2.8,6.5) to (2.8,6.2);
      \draw[-] (2.8,6.2) to (3.1,6.2);
      \draw[-] (3.1,6.2) to (3.1,5.0);
      \draw[-] (3.1,5.0) to (2.8,5.0);
      \draw[-] (2.8,5.0) to (2.8,4.7);
      \draw[-] (2.8,4.7) to (2.5,4.7);
      \draw[-] (2.5,4.7) to (2.5,4.4);
      \draw[-] (2.5,4.4) to (1.9,4.4);
      \draw[-] (1.9,4.4) to (1.9,4.1);
      \draw[-] (1.9,4.1) to (1.0,4.1);
      \tikzstyle{Element}=[draw,minimum width=3mm, minimum height=3mm, node distance=3mm, inner sep=0pt]
      \node[Element,fill=lightgray] at (2.65,4.55) {1};
      \node at (0.8,4.55) {$i$};
      \node (i1) at (3,7) {$i+1$};
      \draw[-] (2.65,6.5) to (i1);
    \end{tikzpicture}\normalsize
    \vspace{-3mm}
    \caption{(WB1) and (WB2)}\label{weakbruhat_p1ton1}
  \end{figure}

If $s_i\not=s_0$, then $\sigma(i)<\sigma(i+1)$ and it is impossible to have $\sigma(i)\ge i$ and $\sigma(i+1)<i+1$ at the same time. Hence, we have following five cases to consider;

 \begin{itemize}
   \item[(WB2)] $\sigma(i)< i$ and $\sigma(i+1) \ge i+1$,
   \item[(WB3)] $\sigma(i)=i$ and $\sigma(i+1) \ge i+1$,
   \item[(WB4)] $\sigma(i)> i$ and $\sigma(i+1) > i+1$,
   \item[(WB5)] $0<\sigma(i)< i$ and $\sigma(i+1) < i+1$,
   \item[(WB6)] $\sigma(i)<0<i$ and $\sigma(i+1)<i+1$.
 \end{itemize}

In the following, we consider each case of (WB2)-(WB6), where we use $(j,k)$ to indicate the box in $\verb"row"\,j$ and $\verb"col"\,k$ in a permutation tableau.

(WB2): Since $\sigma(i)< i$ and $\sigma(i+1) \ge i+1$,  we have $i\in\drop(\sigma)$ and $i+1\in\wex(\sigma)$, and the box $(i+1, i)$ is not in the diagram of $\T$. We can obtain $\mathcal{T}'$ from $\mathcal{T}$ by adding the box $(i+1, i)$ filled with $1$ to $\T$ and exchanging the labels $i$ and $i+1$; see Figure \ref{weakbruhat_p1ton1}.

(WB3): Since $\sigma(i)=i$ and $\sigma(i+1)\ge i+1$, we have $i, i+1\in\wex(\sigma)$ and $\verb"row"\,i$ of $\T$ is a zero row. To obtain $\T'$, delete $\verb"row"\,i$ and insert a new column with only one $1$ at the downmost box that is right next to the rightmost box of $\verb"row"\,(i+1)$ of $\T$ and its corresponding (negative) zero row. Relabel the rows and columns to have a new permutation tableau $\T'$. See Figure \ref{weakbruhat_newcol} and check that $\zeta(\T')=\sigma s_i$.

        \begin{figure}[!ht]
          \begin{tikzpicture}
            \node at (0.5,7.2) {$\T=$};\tiny
            \draw[-] (1.0,8.0) to (1.0,4.1);
            \draw[-] (1.0,8.0) to (1.3,8.0);
            \draw[-] (1.3,8.0) to (1.3,7.7);
            \draw[-] (1.3,7.7) to (1.6,7.7);
            \draw[-] (1.6,7.7) to (1.6,7.4);
            \draw[-] (1.6,7.4) to (1.9,7.4);
            \draw[-] (1.9,7.4) to (1.9,7.1);
            \draw[-] (1.9,7.1) to (2.2,7.1);
            \draw[-] (2.2,7.1) to (2.2,6.8);
            \draw[-] (2.2,6.8) to (2.5,6.8);
            \draw[-] (2.5,6.8) to (2.5,6.5);
            \draw[-] (2.5,6.5) to (2.8,6.5);
            \draw[-] (2.8,6.5) to (2.8,5.3);
            \draw[-] (2.8,5.3) to (2.5,5.3);
            \draw[-] (2.5,5.3) to (2.5,4.4);
            \draw[-] (2.5,4.4) to (1.9,4.4);
            \draw[-] (1.9,4.4) to (1.9,4.1);
            \draw[-] (1.9,4.1) to (1.0,4.1);
            \tikzstyle{Element}=[draw,minimum width=15mm, minimum height=3mm, inner sep=0pt]
            \node[Element,fill=lightgray] [label=left:$i$] at (1.75,5.15) {$0\,\,\,\cdots\,\,\,0$};
            \node[Element,fill=lightgray] [label=left:$i+1$] at (1.75,4.85) {$\ast\,\,\,\cdots\,\,\,\ast$};
          \end{tikzpicture}\normalsize\quad\quad\quad
          \begin{tikzpicture}
            \node at (0.5,7.2) {$\T'=$};\tiny
            \draw[-] (1.0,8.0) to (1.0,4.1);
            \draw[-] (1.0,8.0) to (1.3,8.0);
            \draw[-] (1.3,8.0) to (1.3,7.7);
            \draw[-] (1.3,7.7) to (1.6,7.7);
            \draw[-] (1.6,7.7) to (1.6,7.4);
            \draw[-] (1.6,7.4) to (1.9,7.4);
            \draw[-] (1.9,7.4) to (1.9,7.1);
            \draw[-] (1.9,7.1) to (2.2,7.1);
            \draw[-] (2.2,7.1) to (2.2,6.8);
            \draw[-] (2.2,6.8) to (2.5,6.8);
            \draw[-] (2.5,6.8) to (2.5,6.5);
            \draw[-] (2.5,6.5) to (2.8,6.5);
            \draw[-] (2.8,6.5) to (2.8,6.2);
            \draw[-] (2.8,6.2) to (3.1,6.2);
            \draw[-] (3.1,6.2) to (3.1,5.0);
            \draw[-] (3.1,5.0) to (2.8,5.0);
            \draw[-] (2.8,5.0) to (2.8,4.7);
            \draw[-] (2.8,4.7) to (2.5,4.7);
            \draw[-] (2.5,4.7) to (2.5,4.4);
            \draw[-] (2.5,4.4) to (1.9,4.4);
            \draw[-] (1.9,4.4) to (1.9,4.1);
            \draw[-] (1.9,4.1) to (1.0,4.1);
            \tikzstyle{Element}=[draw,minimum width=15mm, minimum height=3mm, inner sep=0pt]
            \node[Element,fill=lightgray] (0) [label=left:$i$] at (1.75,4.85) {$\ast\,\,\,\cdots\,\,\,\ast$};
            \node (1) at (3.75,4.25) {$\verb"row"\,(i+1)$ of $\T$};
            \draw[-] (0) to (1);
            \node[Element,fill=lightgray] [label=left:$-(i+1)$] at (1.75,6.35) {$0\,\,\,\cdots\,\,\,0$};
            \tikzstyle{Element}=[draw,minimum width=3mm, minimum height=12mm, inner sep=0pt]
            \node[Element,fill=lightgray] at (2.65,5.6) {};
            \node at (2.65,6.05) {0}; \node at (2.65,5.7) {$\vdots$}; \node at (2.65, 5.15) {0};
            \tikzstyle{Element}=[draw,minimum width=3mm, minimum height=3mm, inner sep=0pt]
            \node[Element,fill=lightgray] at (2.65,4.85) {1};
            \node[Element,fill=lightgray] at (2.65,6.35) {0};
          \end{tikzpicture}\normalsize\vspace{-3mm}
          \caption{(WB3)}\label{weakbruhat_newcol}
        \end{figure}

(WB4): Since $\sigma(i)> i$ and $\sigma(i+1) > i+1$, $\verb"row"\,i$ and $\verb"row"\,(i+1)$ are two consecutive rows in $\T$.
  \begin{itemize}
    \item[(WB4-1)] If two zigzag paths from $\verb"row"\,i$ and $\verb"row"\,(i+1)$ meet at a box with $1$, (note that it is impossible to cross at a box with $0$ because $\sigma(i)<\sigma(i+1)$) then replace the first such $1$ by $0$, which must be of type $EE$, so that we obtain $\mathcal{T}'$. See Figure \ref{weakbruhat_1to0ee}.
    \item[(WB4-2)] If two zigzag paths from $\verb"row"\,i$ and $\verb"row"\,(i+1)$ do not meet, then we can apply Lemma~\ref{rule} to the smallest rectangular part containing the leftmost $1$'s of \verb"row"\,$i$ and \verb"row"\,$(i+1)$ to obtain $\T'$: See Figure \ref{weakbruhat_1to0ee}.
  \end{itemize}

        \begin{figure}[!ht]
          \begin{tikzpicture}
            \node at (0.5,6.6) {$\T=$};\tiny
            \draw[-] (1.0,8.0) to (1.0,4.1);
            \draw[-] (1.0,8.0) to (1.3,8.0);
            \draw[-] (1.3,8.0) to (1.3,7.7);
            \draw[-] (1.3,7.7) to (1.6,7.7);
            \draw[-] (1.6,7.7) to (1.6,7.4);
            \draw[-] (1.6,7.4) to (1.9,7.4);
            \draw[-] (1.9,7.4) to (1.9,7.1);
            \draw[-] (1.9,7.1) to (2.2,7.1);
            \draw[-] (2.2,7.1) to (2.2,6.8);
            \draw[-] (2.2,6.8) to (2.5,6.8);
            \draw[-] (2.5,6.8) to (2.5,6.5);
            \draw[-] (2.5,6.5) to (2.8,6.5);
            \draw[-] (2.8,6.5) to (2.8,6.2);
            \draw[-] (2.8,6.2) to (3.1,6.2);
            \draw[-] (3.1,6.2) to (3.1,5.0);
            \draw[-] (3.1,5.0) to (2.8,5.0);
            \draw[-] (2.8,5.0) to (2.8,4.7);
            \draw[-] (2.8,4.7) to (2.5,4.7);
            \draw[-] (2.5,4.7) to (2.5,4.4);
            \draw[-] (2.5,4.4) to (1.9,4.4);
            \draw[-] (1.9,4.4) to (1.9,4.1);
            \draw[-] (1.9,4.1) to (1.0,4.1);
            \tikzstyle{Element}=[draw,minimum width=3mm, minimum height=3mm, inner sep=0pt]
            \node[Element,fill=lightgray] at (2.05,5.75) {1};
            \node at (0.8,6.05) {$i$};
            \draw[-] (1.0,6.05) to (2.05,6.05);
            \draw[-] (2.05,6.05) to (2.05,5.75);
            \draw[->] (2.05,5.75) to (2.5,5.75);
            \node at (0.6,5.75) {$i+1$};
            \draw[-,dashed] (1.0,5.75) to (2.05,5.75);
            \draw[->,dashed] (2.05,5.75) to (2.05,5.0);
          \end{tikzpicture}\normalsize\quad
          \begin{tikzpicture}
            \node at (0.5,6.6) {$\T'=$};\tiny
            \draw[-] (1.0,8.0) to (1.0,4.1);
            \draw[-] (1.0,8.0) to (1.3,8.0);
            \draw[-] (1.3,8.0) to (1.3,7.7);
            \draw[-] (1.3,7.7) to (1.6,7.7);
            \draw[-] (1.6,7.7) to (1.6,7.4);
            \draw[-] (1.6,7.4) to (1.9,7.4);
            \draw[-] (1.9,7.4) to (1.9,7.1);
            \draw[-] (1.9,7.1) to (2.2,7.1);
            \draw[-] (2.2,7.1) to (2.2,6.8);
            \draw[-] (2.2,6.8) to (2.5,6.8);
            \draw[-] (2.5,6.8) to (2.5,6.5);
            \draw[-] (2.5,6.5) to (2.8,6.5);
            \draw[-] (2.8,6.5) to (2.8,6.2);
            \draw[-] (2.8,6.2) to (3.1,6.2);
            \draw[-] (3.1,6.2) to (3.1,5.0);
            \draw[-] (3.1,5.0) to (2.8,5.0);
            \draw[-] (2.8,5.0) to (2.8,4.7);
            \draw[-] (2.8,4.7) to (2.5,4.7);
            \draw[-] (2.5,4.7) to (2.5,4.4);
            \draw[-] (2.5,4.4) to (1.9,4.4);
            \draw[-] (1.9,4.4) to (1.9,4.1);
            \draw[-] (1.9,4.1) to (1.0,4.1);
            \tikzstyle{Element}=[draw,minimum width=3mm, minimum height=3mm, inner sep=0pt]
            \node[Element,fill=lightgray] at (2.05,5.75) {$_{\texttt{EE}}$};
            \node at (0.8,6.05) {$i$};
            \draw[-] (1.0,6.05) to (2.05,6.05);
            \draw[->] (2.05,6.05) to (2.05,5.0);
            \node at (0.6,5.75) {$i+1$};
            \draw[->,dashed] (1.0,5.75) to (2.5,5.75);
          \end{tikzpicture}\normalsize\quad\quad\quad\quad
          \begin{tikzpicture}
            \node at (0.5,7.2) {$\T=$};\tiny
            \draw[-] (1.0,8.0) to (1.0,4.1);
            \draw[-] (1.0,8.0) to (1.3,8.0);
            \draw[-] (1.3,8.0) to (1.3,7.7);
            \draw[-] (1.3,7.7) to (1.6,7.7);
            \draw[-] (1.6,7.7) to (1.6,7.4);
            \draw[-] (1.6,7.4) to (1.9,7.4);
            \draw[-] (1.9,7.4) to (1.9,7.1);
            \draw[-] (1.9,7.1) to (2.2,7.1);
            \draw[-] (2.2,7.1) to (2.2,6.8);
            \draw[-] (2.2,6.8) to (2.5,6.8);
            \draw[-] (2.5,6.8) to (2.5,6.5);
            \draw[-] (2.5,6.5) to (2.8,6.5);
            \draw[-] (2.8,6.5) to (2.8,6.2);
            \draw[-] (2.8,6.2) to (3.1,6.2);
            \draw[-] (3.1,6.2) to (3.1,5.0);
            \draw[-] (3.1,5.0) to (2.8,5.0);
            \draw[-] (2.8,5.0) to (2.8,4.7);
            \draw[-] (2.8,4.7) to (2.5,4.7);
            \draw[-] (2.5,4.7) to (2.5,4.4);
            \draw[-] (2.5,4.4) to (1.9,4.4);
            \draw[-] (1.9,4.4) to (1.9,4.1);
            \draw[-] (1.9,4.1) to (1.0,4.1);
            \tikzstyle{Element}=[draw,minimum width=3mm, minimum height=3mm, inner sep=0pt]
            \node at (0.7,5.75) {$i$}; \node at (0.6,5.45) {$i+1$}; \node at (1.7,5.6) {$\cdot$}; \node at (1.75,5.6) {$\cdot$}; \node at (1.8,5.6) {$\cdot$}; \node at (2.3,5.6) {$\cdot$}; \node at (2.35,5.6) {$\cdot$}; \node at (2.4,5.6) {$\cdot$};
            \node[Element,fill=lightgray] at (1.45,5.75) {0}; \node[Element,fill=lightgray] at (2.05,5.75) {0}; \node[Element,fill=lightgray] at (2.65,5.75) {1};
            \node[Element,fill=lightgray] at (1.45,5.45) {1}; \node[Element,fill=lightgray] at (2.05,5.45) {1}; \node[Element,fill=lightgray] at (2.65,5.45) {1};
          \end{tikzpicture}\normalsize\quad
          \begin{tikzpicture}
            \node at (0.5,7.2) {$\T'=$};\tiny
            \draw[-] (1.0,8.0) to (1.0,4.1);
            \draw[-] (1.0,8.0) to (1.3,8.0);
            \draw[-] (1.3,8.0) to (1.3,7.7);
            \draw[-] (1.3,7.7) to (1.6,7.7);
            \draw[-] (1.6,7.7) to (1.6,7.4);
            \draw[-] (1.6,7.4) to (1.9,7.4);
            \draw[-] (1.9,7.4) to (1.9,7.1);
            \draw[-] (1.9,7.1) to (2.2,7.1);
            \draw[-] (2.2,7.1) to (2.2,6.8);
            \draw[-] (2.2,6.8) to (2.5,6.8);
            \draw[-] (2.5,6.8) to (2.5,6.5);
            \draw[-] (2.5,6.5) to (2.8,6.5);
            \draw[-] (2.8,6.5) to (2.8,6.2);
            \draw[-] (2.8,6.2) to (3.1,6.2);
            \draw[-] (3.1,6.2) to (3.1,5.0);
            \draw[-] (3.1,5.0) to (2.8,5.0);
            \draw[-] (2.8,5.0) to (2.8,4.7);
            \draw[-] (2.8,4.7) to (2.5,4.7);
            \draw[-] (2.5,4.7) to (2.5,4.4);
            \draw[-] (2.5,4.4) to (1.9,4.4);
            \draw[-] (1.9,4.4) to (1.9,4.1);
            \draw[-] (1.9,4.1) to (1.0,4.1);
            \tikzstyle{Element}=[draw,minimum width=3mm, minimum height=3mm, inner sep=0pt]
            \node at (0.7,5.75) {$i$}; \node at (0.6,5.45) {$i+1$}; \node at (1.7,5.6) {$\cdot$}; \node at (1.75,5.6) {$\cdot$}; \node at (1.8,5.6) {$\cdot$}; \node at (2.3,5.6) {$\cdot$}; \node at (2.35,5.6) {$\cdot$}; \node at (2.4,5.6) {$\cdot$};
            \node[Element,fill=lightgray] at (1.45,5.75) {1}; \node[Element,fill=lightgray] at (2.05,5.75) {1}; \node[Element,fill=lightgray] at (2.65,5.75) {1};
            \node[Element,fill=lightgray] at (1.45,5.45) {$_{\texttt{EE}}$}; \node[Element,fill=lightgray] at (2.05,5.45) {0}; \node[Element,fill=lightgray] at (2.65,5.45) {0};
          \end{tikzpicture}\normalsize\vspace{-3mm}
          \caption{(WB4-1) and (WB4-2)}\label{weakbruhat_1to0ee}
        \end{figure}
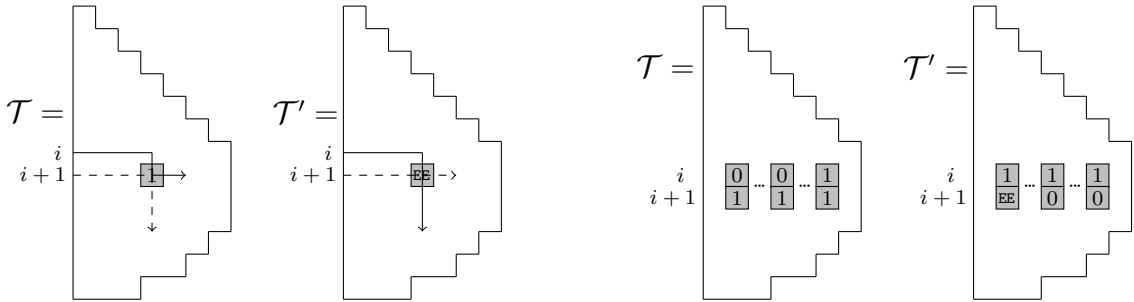

(WB5): Since $0<\sigma(i)<i$ and $0<\sigma(i+1)<i+1$, $\verb"col"\,i$ and $\verb"col"\,(i+1)$ have  $0$'s in their diagonals.
  \begin{itemize}
    \item[(WB5-1)] If two zigzag paths from $\verb"col"\,i$ and $\verb"col"\,(i+1)$ meet at a box with $1$, then replace the first such $1$ by $0$, which must be of type $NN$, so that we obtain $\mathcal{T}'$. See Figure \ref{weakbruhat_1to0nn}.
    \item[(WB5-2)] If two zigzag paths from $\verb"col"\,i$ and $\verb"col"\,(i+1)$ do not meet, then apply Lemma~\ref{rule} to the smallest rectangular part containing the topmost $1$'s of \verb"col"\,$i$ and \verb"col"\,$(i+1)$ to obtain $\T'$; see Figure \ref{weakbruhat_1to0nn}.
  \end{itemize}

        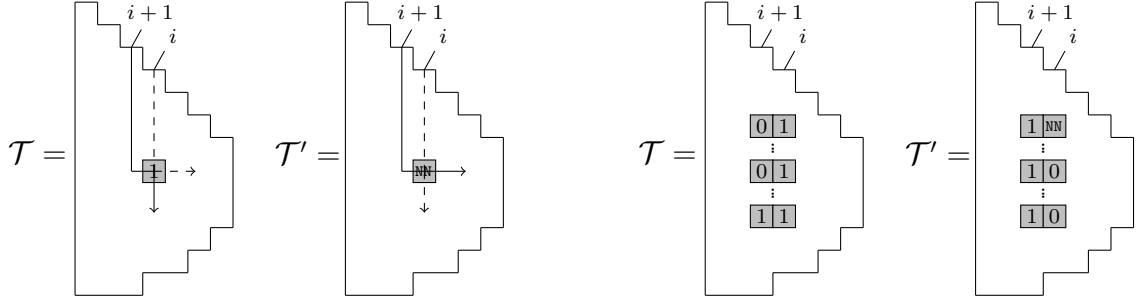
\begin{figure}[!ht]
          \begin{tikzpicture}
            \node at (0.5,6.0) {$\T=$};\tiny
            \draw[-] (1.0,8.0) to (1.0,4.1);
            \draw[-] (1.0,8.0) to (1.3,8.0);
            \draw[-] (1.3,8.0) to (1.3,7.7);
            \draw[-] (1.3,7.7) to (1.6,7.7);
            \draw[-] (1.6,7.7) to (1.6,7.4);
            \draw[-] (1.6,7.4) to (1.9,7.4);
            \draw[-] (1.9,7.4) to (1.9,7.1);
            \draw[-] (1.9,7.1) to (2.2,7.1);
            \draw[-] (2.2,7.1) to (2.2,6.8);
            \draw[-] (2.2,6.8) to (2.5,6.8);
            \draw[-] (2.5,6.8) to (2.5,6.5);
            \draw[-] (2.5,6.5) to (2.8,6.5);
            \draw[-] (2.8,6.5) to (2.8,6.2);
            \draw[-] (2.8,6.2) to (3.1,6.2);
            \draw[-] (3.1,6.2) to (3.1,5.0);
            \draw[-] (3.1,5.0) to (2.8,5.0);
            \draw[-] (2.8,5.0) to (2.8,4.7);
            \draw[-] (2.8,4.7) to (2.5,4.7);
            \draw[-] (2.5,4.7) to (2.5,4.4);
            \draw[-] (2.5,4.4) to (1.9,4.4);
            \draw[-] (1.9,4.4) to (1.9,4.1);
            \draw[-] (1.9,4.1) to (1.0,4.1);
            \tikzstyle{Element}=[draw,minimum width=3mm, minimum height=3mm, inner sep=0pt]
            \node[Element,fill=lightgray] at (2.05,5.75) {1};
            \node (i1) at (2.0,7.85) {$i+1$}; \node (i) at (2.3,7.55) {$i$};
            \draw[-] (1.75,7.4) to (i1);
            \draw[-] (2.05,7.1) to (i);
            \draw[-] (1.75,7.4) to (1.75,5.75);
            \draw[-] (1.75,5.75) to (2.05,5.75);
            \draw[->] (2.05,5.75) to (2.05,5.2);
            \draw[-,dashed] (2.05,7.1) to (2.05,5.75);
            \draw[->,dashed] (2.05,5.75) to (2.6,5.75);
          \end{tikzpicture}\normalsize\quad
          \begin{tikzpicture}
            \node at (0.5,6.0) {$\T'=$};\tiny
            \draw[-] (1.0,8.0) to (1.0,4.1);
            \draw[-] (1.0,8.0) to (1.3,8.0);
            \draw[-] (1.3,8.0) to (1.3,7.7);
            \draw[-] (1.3,7.7) to (1.6,7.7);
            \draw[-] (1.6,7.7) to (1.6,7.4);
            \draw[-] (1.6,7.4) to (1.9,7.4);
            \draw[-] (1.9,7.4) to (1.9,7.1);
            \draw[-] (1.9,7.1) to (2.2,7.1);
            \draw[-] (2.2,7.1) to (2.2,6.8);
            \draw[-] (2.2,6.8) to (2.5,6.8);
            \draw[-] (2.5,6.8) to (2.5,6.5);
            \draw[-] (2.5,6.5) to (2.8,6.5);
            \draw[-] (2.8,6.5) to (2.8,6.2);
            \draw[-] (2.8,6.2) to (3.1,6.2);
            \draw[-] (3.1,6.2) to (3.1,5.0);
            \draw[-] (3.1,5.0) to (2.8,5.0);
            \draw[-] (2.8,5.0) to (2.8,4.7);
            \draw[-] (2.8,4.7) to (2.5,4.7);
            \draw[-] (2.5,4.7) to (2.5,4.4);
            \draw[-] (2.5,4.4) to (1.9,4.4);
            \draw[-] (1.9,4.4) to (1.9,4.1);
            \draw[-] (1.9,4.1) to (1.0,4.1);
            \tikzstyle{Element}=[draw,minimum width=3mm, minimum height=3mm, inner sep=0pt]
            \node[Element,fill=lightgray] at (2.05,5.75) {$_{\texttt{NN}}$};
            \node (i1) at (2.0,7.85) {$i+1$}; \node (i) at (2.3,7.55) {$i$};
            \draw[-] (1.75,7.4) to (i1);
            \draw[-] (2.05,7.1) to (i);
            \draw[-] (1.75,7.4) to (1.75,5.75);
            \draw[-] (1.75,5.75) to (2.05,5.75);
            \draw[->,dashed] (2.05,5.75) to (2.05,5.2);
            \draw[-,dashed] (2.05,7.1) to (2.05,5.75);
            \draw[->] (2.05,5.75) to (2.6,5.75);
          \end{tikzpicture}\normalsize\quad\quad\quad\quad
          \begin{tikzpicture}
            \node at (0.5,6.0) {$\T=$};\tiny
            \draw[-] (1.0,8.0) to (1.0,4.1);
            \draw[-] (1.0,8.0) to (1.3,8.0);
            \draw[-] (1.3,8.0) to (1.3,7.7);
            \draw[-] (1.3,7.7) to (1.6,7.7);
            \draw[-] (1.6,7.7) to (1.6,7.4);
            \draw[-] (1.6,7.4) to (1.9,7.4);
            \draw[-] (1.9,7.4) to (1.9,7.1);
            \draw[-] (1.9,7.1) to (2.2,7.1);
            \draw[-] (2.2,7.1) to (2.2,6.8);
            \draw[-] (2.2,6.8) to (2.5,6.8);
            \draw[-] (2.5,6.8) to (2.5,6.5);
            \draw[-] (2.5,6.5) to (2.8,6.5);
            \draw[-] (2.8,6.5) to (2.8,6.2);
            \draw[-] (2.8,6.2) to (3.1,6.2);
            \draw[-] (3.1,6.2) to (3.1,5.0);
            \draw[-] (3.1,5.0) to (2.8,5.0);
            \draw[-] (2.8,5.0) to (2.8,4.7);
            \draw[-] (2.8,4.7) to (2.5,4.7);
            \draw[-] (2.5,4.7) to (2.5,4.4);
            \draw[-] (2.5,4.4) to (1.9,4.4);
            \draw[-] (1.9,4.4) to (1.9,4.1);
            \draw[-] (1.9,4.1) to (1.0,4.1);
            \tikzstyle{Element}=[draw,minimum width=3mm, minimum height=3mm, inner sep=0pt]
            \node[Element,fill=lightgray] at (1.75,6.35) {0}; \node[Element,fill=lightgray] at (2.05,6.35) {1};
            \node at (1.9,6.1) {$\cdot$}; \node at (1.9,6.05) {$\cdot$}; \node at (1.9,6.0) {$\cdot$};
            \node[Element,fill=lightgray] at (1.75,5.75) {0}; \node[Element,fill=lightgray] at (2.05,5.75) {1};
            \node at (1.9,5.5) {$\cdot$}; \node at (1.9,5.45) {$\cdot$}; \node at (1.9,5.4) {$\cdot$};
            \node[Element,fill=lightgray] at (1.75,5.15) {1}; \node[Element,fill=lightgray] at (2.05,5.15) {1};
            \node (i1) at (2.0,7.85) {$i+1$}; \node (i) at (2.3,7.55) {$i$};
            \draw[-] (1.75,7.4) to (i1);
            \draw[-] (2.05,7.1) to (i);
          \end{tikzpicture}\normalsize\quad
          \begin{tikzpicture}
            \node at (0.5,6.0) {$\T'=$};\tiny
            \draw[-] (1.0,8.0) to (1.0,4.1);
            \draw[-] (1.0,8.0) to (1.3,8.0);
            \draw[-] (1.3,8.0) to (1.3,7.7);
            \draw[-] (1.3,7.7) to (1.6,7.7);
            \draw[-] (1.6,7.7) to (1.6,7.4);
            \draw[-] (1.6,7.4) to (1.9,7.4);
            \draw[-] (1.9,7.4) to (1.9,7.1);
            \draw[-] (1.9,7.1) to (2.2,7.1);
            \draw[-] (2.2,7.1) to (2.2,6.8);
            \draw[-] (2.2,6.8) to (2.5,6.8);
            \draw[-] (2.5,6.8) to (2.5,6.5);
            \draw[-] (2.5,6.5) to (2.8,6.5);
            \draw[-] (2.8,6.5) to (2.8,6.2);
            \draw[-] (2.8,6.2) to (3.1,6.2);
            \draw[-] (3.1,6.2) to (3.1,5.0);
            \draw[-] (3.1,5.0) to (2.8,5.0);
            \draw[-] (2.8,5.0) to (2.8,4.7);
            \draw[-] (2.8,4.7) to (2.5,4.7);
            \draw[-] (2.5,4.7) to (2.5,4.4);
            \draw[-] (2.5,4.4) to (1.9,4.4);
            \draw[-] (1.9,4.4) to (1.9,4.1);
            \draw[-] (1.9,4.1) to (1.0,4.1);
            \tikzstyle{Element}=[draw,minimum width=3mm, minimum height=3mm, inner sep=0pt]
            \node[Element,fill=lightgray] at (1.75,6.35) {1}; \node[Element,fill=lightgray] at (2.05,6.35) {$_{\texttt{NN}}$};
            \node at (1.9,6.1) {$\cdot$}; \node at (1.9,6.05) {$\cdot$}; \node at (1.9,6.0) {$\cdot$};
            \node[Element,fill=lightgray] at (1.75,5.75) {1}; \node[Element,fill=lightgray] at (2.05,5.75) {0};
            \node at (1.9,5.5) {$\cdot$}; \node at (1.9,5.45) {$\cdot$}; \node at (1.9,5.4) {$\cdot$};
            \node[Element,fill=lightgray] at (1.75,5.15) {1}; \node[Element,fill=lightgray] at (2.05,5.15) {0};
            \node (i1) at (2.0,7.85) {$i+1$}; \node (i) at (2.3,7.55) {$i$};
            \draw[-] (1.75,7.4) to (i1);
            \draw[-] (2.05,7.1) to (i);
          \end{tikzpicture}\normalsize\vspace{-3mm}
          \caption{(WB5-1) and (WB5-2)}\label{weakbruhat_1to0nn}
        \end{figure}

(WB6): Since $\sigma(i)<0<i$ and $\sigma(i+1)<i+1$, $\verb"row"\,(-i)$ and $\verb"row"\,(-(i+1))$ are two consecutive rows and there is $1$ in the diagonal of $\verb"row"\,(-i)$. There are two cases depending on the sign of $\sigma(i+1)$:
  \begin{itemize}
    \item[(WB6-1)] Let $\sigma(i+1)<0$, then, $\verb"row"\,(-(i+1))$ has 1 in its diagonal. If two zigzag paths from $\verb"row"\,(-i)$ and $\verb"row"\,(-(i+1))$ meet at a box with $1$, then replace the first such $1$ by $0$, which must be of type ${EE}$ so that we obtain $\mathcal{T}'$; see Figure \ref{weakbruhat_1to0ee_0hinge}. If two zigzag paths from $\verb"row"\,(-i)$ and $\verb"row"\,(-(i+1))$ do not meet, then apply Lemma~\ref{rule} to the smallest rectangular part containing the leftmost $1$'s in $\verb"row"\,(-i)$ and $\verb"row"\,(-(i+1))$ to obtain $\T'$.
    \item[(WB6-2)] Let $\sigma(i+1)>0$, then $\verb"row"\,(-(i+1))$ has 0 in its diagonal and two zigzag paths from $\verb"row"\,(-i)$ and $\verb"row"\,(-(i+1))$ do not meet. We apply Lemma~\ref{rule} to the smallest rectangular part containing the leftmost $1$'s in $\verb"row"\,(-i)$ and $\verb"row"\,(-(i+1))$ to obtain $\T'$, where we assume that there is a (imaginary) $1$ right above the diagonal of $\verb"row"\,(-i)$. See Figure~\ref{weakbruhat_1to0ee_0hinge}.
        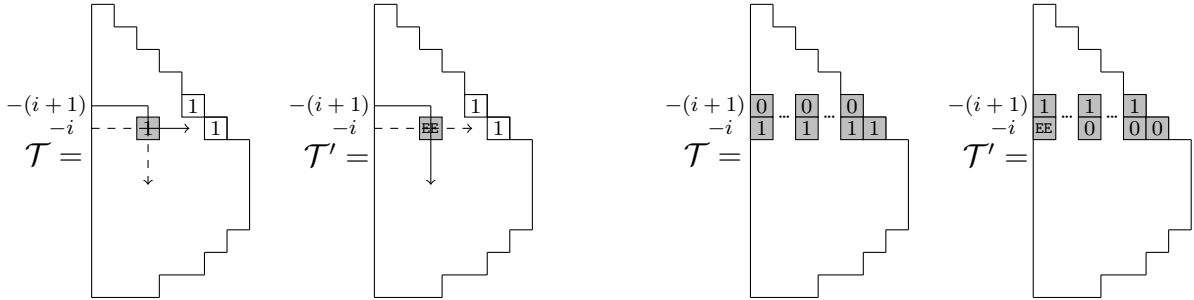
\begin{figure}[!ht]
          \begin{tikzpicture}
            \node at (0.5,6.0) {$\T=$};\tiny
            \draw[-] (1.0,8.0) to (1.0,4.1);
            \draw[-] (1.0,8.0) to (1.3,8.0);
            \draw[-] (1.3,8.0) to (1.3,7.7);
            \draw[-] (1.3,7.7) to (1.6,7.7);
            \draw[-] (1.6,7.7) to (1.6,7.4);
            \draw[-] (1.6,7.4) to (1.9,7.4);
            \draw[-] (1.9,7.4) to (1.9,7.1);
            \draw[-] (1.9,7.1) to (2.2,7.1);
            \draw[-] (2.2,7.1) to (2.2,6.8);
            \draw[-] (2.2,6.8) to (2.5,6.8);
            \draw[-] (2.5,6.8) to (2.5,6.5);
            \draw[-] (2.5,6.5) to (2.8,6.5);
            \draw[-] (2.8,6.5) to (2.8,6.2);
            \draw[-] (2.8,6.2) to (3.1,6.2);
            \draw[-] (3.1,6.2) to (3.1,5.0);
            \draw[-] (3.1,5.0) to (2.8,5.0);
            \draw[-] (2.8,5.0) to (2.8,4.7);
            \draw[-] (2.8,4.7) to (2.5,4.7);
            \draw[-] (2.5,4.7) to (2.5,4.4);
            \draw[-] (2.5,4.4) to (1.9,4.4);
            \draw[-] (1.9,4.4) to (1.9,4.1);
            \draw[-] (1.9,4.1) to (1.0,4.1);
            \tikzstyle{Element}=[draw,minimum width=3mm, minimum height=3mm, inner sep=0pt]
            \node[Element] at (2.35,6.65) {1};
            \node[Element] at (2.65,6.35) {1};
            \node[Element,fill=lightgray] at (1.75,6.35) {1};
            \node at (0.4,6.65) {$-(i+1)$};
            \draw[-] (1.0,6.65) to (1.75,6.65);
            \draw[-] (1.75,6.65) to (1.75,6.35);
            \draw[->] (1.75,6.35) to (2.3,6.35);
            \node at (0.6,6.35) {$-i$};
            \draw[-,dashed] (1.0,6.35) to (1.75,6.35);
            \draw[->,dashed] (1.75,6.35) to (1.75,5.6);
          \end{tikzpicture}\normalsize\quad
          \begin{tikzpicture}
            \node at (0.5,6.0) {$\T'=$};\tiny
            \draw[-] (1.0,8.0) to (1.0,4.1);
            \draw[-] (1.0,8.0) to (1.3,8.0);
            \draw[-] (1.3,8.0) to (1.3,7.7);
            \draw[-] (1.3,7.7) to (1.6,7.7);
            \draw[-] (1.6,7.7) to (1.6,7.4);
            \draw[-] (1.6,7.4) to (1.9,7.4);
            \draw[-] (1.9,7.4) to (1.9,7.1);
            \draw[-] (1.9,7.1) to (2.2,7.1);
            \draw[-] (2.2,7.1) to (2.2,6.8);
            \draw[-] (2.2,6.8) to (2.5,6.8);
            \draw[-] (2.5,6.8) to (2.5,6.5);
            \draw[-] (2.5,6.5) to (2.8,6.5);
            \draw[-] (2.8,6.5) to (2.8,6.2);
            \draw[-] (2.8,6.2) to (3.1,6.2);
            \draw[-] (3.1,6.2) to (3.1,5.0);
            \draw[-] (3.1,5.0) to (2.8,5.0);
            \draw[-] (2.8,5.0) to (2.8,4.7);
            \draw[-] (2.8,4.7) to (2.5,4.7);
            \draw[-] (2.5,4.7) to (2.5,4.4);
            \draw[-] (2.5,4.4) to (1.9,4.4);
            \draw[-] (1.9,4.4) to (1.9,4.1);
            \draw[-] (1.9,4.1) to (1.0,4.1);
            \tikzstyle{Element}=[draw,minimum width=3mm, minimum height=3mm, inner sep=0pt]
            \node[Element] at (2.35,6.65) {1};
            \node[Element] at (2.65,6.35) {1};
            \node[Element,fill=lightgray] at (1.75,6.35) {$_{\texttt{EE}}$};
            \node at (0.4,6.65) {$-(i+1)$};
            \draw[-] (1.0,6.65) to (1.75,6.65);
            \draw[-] (1.75,6.65) to (1.75,6.35);
            \draw[->,dashed] (1.75,6.35) to (2.3,6.35);
            \node at (0.6,6.35) {$-i$};
            \draw[-,dashed] (1.0,6.35) to (1.75,6.35);
            \draw[->] (1.75,6.35) to (1.75,5.6);
          \end{tikzpicture}\normalsize\quad\quad\quad\quad
          \begin{tikzpicture}
            \node at (0.5,6.0) {$\T=$};\tiny
            \draw[-] (1.0,8.0) to (1.0,4.1);
            \draw[-] (1.0,8.0) to (1.3,8.0);
            \draw[-] (1.3,8.0) to (1.3,7.7);
            \draw[-] (1.3,7.7) to (1.6,7.7);
            \draw[-] (1.6,7.7) to (1.6,7.4);
            \draw[-] (1.6,7.4) to (1.9,7.4);
            \draw[-] (1.9,7.4) to (1.9,7.1);
            \draw[-] (1.9,7.1) to (2.2,7.1);
            \draw[-] (2.2,7.1) to (2.2,6.8);
            \draw[-] (2.2,6.8) to (2.5,6.8);
            \draw[-] (2.5,6.8) to (2.5,6.5);
            \draw[-] (2.5,6.5) to (2.8,6.5);
            \draw[-] (2.8,6.5) to (2.8,6.2);
            \draw[-] (2.8,6.2) to (3.1,6.2);
            \draw[-] (3.1,6.2) to (3.1,5.0);
            \draw[-] (3.1,5.0) to (2.8,5.0);
            \draw[-] (2.8,5.0) to (2.8,4.7);
            \draw[-] (2.8,4.7) to (2.5,4.7);
            \draw[-] (2.5,4.7) to (2.5,4.4);
            \draw[-] (2.5,4.4) to (1.9,4.4);
            \draw[-] (1.9,4.4) to (1.9,4.1);
            \draw[-] (1.9,4.1) to (1.0,4.1);
            \tikzstyle{Element}=[draw,minimum width=3mm, minimum height=3mm, inner sep=0pt]
            \node at (0.4,6.65) {$-(i+1)$}; \node at (0.6,6.35) {$-i$};
            \node at (1.4,6.5) {$\cdot$}; \node at (1.45,6.5) {$\cdot$}; \node at (1.5,6.5) {$\cdot$}; \node at (2.0,6.5) {$\cdot$}; \node at (2.05,6.5) {$\cdot$}; \node at (2.1,6.5) {$\cdot$};
            \node[Element,fill=lightgray] at (1.15,6.65) {0}; \node[Element,fill=lightgray] at (1.75,6.65) {0}; \node[Element,fill=lightgray] at (2.35,6.65) {0};
            \node[Element,fill=lightgray] at (1.15,6.35) {1}; \node[Element,fill=lightgray] at (1.75,6.35) {1}; \node[Element,fill=lightgray] at (2.35,6.35) {1}; \node[Element,fill=lightgray] at (2.65,6.35) {1};
          \end{tikzpicture}\normalsize\quad
          \begin{tikzpicture}
            \node at (0.5,6.0) {$\T'=$};\tiny
            \draw[-] (1.0,8.0) to (1.0,4.1);
            \draw[-] (1.0,8.0) to (1.3,8.0);
            \draw[-] (1.3,8.0) to (1.3,7.7);
            \draw[-] (1.3,7.7) to (1.6,7.7);
            \draw[-] (1.6,7.7) to (1.6,7.4);
            \draw[-] (1.6,7.4) to (1.9,7.4);
            \draw[-] (1.9,7.4) to (1.9,7.1);
            \draw[-] (1.9,7.1) to (2.2,7.1);
            \draw[-] (2.2,7.1) to (2.2,6.8);
            \draw[-] (2.2,6.8) to (2.5,6.8);
            \draw[-] (2.5,6.8) to (2.5,6.5);
            \draw[-] (2.5,6.5) to (2.8,6.5);
            \draw[-] (2.8,6.5) to (2.8,6.2);
            \draw[-] (2.8,6.2) to (3.1,6.2);
            \draw[-] (3.1,6.2) to (3.1,5.0);
            \draw[-] (3.1,5.0) to (2.8,5.0);
            \draw[-] (2.8,5.0) to (2.8,4.7);
            \draw[-] (2.8,4.7) to (2.5,4.7);
            \draw[-] (2.5,4.7) to (2.5,4.4);
            \draw[-] (2.5,4.4) to (1.9,4.4);
            \draw[-] (1.9,4.4) to (1.9,4.1);
            \draw[-] (1.9,4.1) to (1.0,4.1);
            \tikzstyle{Element}=[draw,minimum width=3mm, minimum height=3mm, inner sep=0pt]
            \node at (0.4,6.65) {$-(i+1)$}; \node at (0.6,6.35) {$-i$};
            \node at (1.4,6.5) {$\cdot$}; \node at (1.45,6.5) {$\cdot$}; \node at (1.5,6.5) {$\cdot$}; \node at (2.0,6.5) {$\cdot$}; \node at (2.05,6.5) {$\cdot$}; \node at (2.1,6.5) {$\cdot$};
            \node[Element,fill=lightgray] at (1.15,6.65) {1}; \node[Element,fill=lightgray] at (1.75,6.65) {1}; \node[Element,fill=lightgray] at (2.35,6.65) {1};
            \node[Element,fill=lightgray] at (1.15,6.35) {$_{\texttt{EE}}$}; \node[Element,fill=lightgray] at (1.75,6.35) {0}; \node[Element,fill=lightgray] at (2.35,6.35) {0}; \node[Element,fill=lightgray] at (2.65,6.35) {0};
          \end{tikzpicture}\normalsize\vspace{-3mm}
          \caption{(WB6-1) and (WB6-2)}\label{weakbruhat_1to0ee_0hinge}
        \end{figure}
  \end{itemize}

\begin{remark}\label{zero_column}
  In cases (WB5-1), (WB5-2), and (WB6-2), it is possible for us to have a zero column \verb"col"\,$i$, and hence a zero row \verb"row"\,$(-i)$  as a result. In this case, delete \verb"col"\,$i$ and \verb"row"\,$(-i)$ and insert the new zero row \verb"row"\,$i$ instead in the resulting permutation tableau.
  For example, consider $\sigma=-4,1,2,-3\in\sym_{4}^{B}$ and $\sigma'=\sigma s_{2}\rhd \sigma$. This is case (WB5-1), and \verb"col"\,$2$ in $\T'$ is a zero column and \verb"row"\,$(-2)$ is a zero row. Thus we insert \verb"row"\,$2$ with all $0$'s, while deleting \verb"col"\,$2$ and \verb"row"\,$(-2)$ to make $\T'$ a \emph{permutation tableau}; see Figure~\ref{newrowex}.
  \begin{figure}[!ht]
    \begin{tikzpicture}
      \tiny
      \tikzstyle{Element} = [draw, minimum width=4mm, minimum height=4mm, node distance=4mm, inner sep=0pt]
      \node[Element] [label=left:-4] at (1,3) {1};
      \node[Element] [label=left:-3] at (1,2.6) {\texttt{EE}};
      \node[Element] [label=left:-2] at (1,2.2) {\texttt{EE}};
      \node[Element] [label=left:-1] at (1,1.8) {1};

      \node[Element] at (1.4,2.6) {0};
      \node[Element] at (1.4,2.2) {0};
      \node[Element] at (1.4,1.8) {1};

      \node[Element] at (1.8,2.2) {0};
      \node[Element] at (1.8,1.8) {1};

      \node[Element] at (2.2,1.8) {1};
      \normalsize
      \node at (0,2.4) {$\T=$};
      \node (b) at (1.1,0.7) {$\sigma = -4,1,2,-3$};
      \draw[<->] (1.1,1.5) to (b);
      \node at (4,2) {$\longrightarrow$};
    \end{tikzpicture}\quad\quad
    \begin{tikzpicture}
      \tiny
      \tikzstyle{Element} = [draw, minimum width=4mm, minimum height=4mm, node distance=4mm, inner sep=0pt]
      \node[Element] [label=left:-4] at (1,3) {1};
      \node[Element] [label=left:-3] at (1,2.6) {\texttt{EE}};
      \node[Element,fill=lightgray] [label=left:\textbf{-2}] at (1,2.2) {\texttt{EE}};
      \node[Element] [label=left:-1] at (1,1.8) {1};

      \node[Element] at (1.4,2.6) {0};
      \node[Element,fill=lightgray] at (1.4,2.2) {0};
      \node[Element] at (1.4,1.8) {1};

      \node[Element,fill=lightgray] at (1.8,2.2) {0};
      \node[Element,fill=lightgray] at (1.8,1.8) {\texttt{NN}};

      \node[Element] at (2.2,1.8) {1};
      \normalsize
      \node at (0,2.4) {$\T'=$};
      \node (b) at (3.7,0.7) {$\sigma' = -4,2,1,-3$};
      \draw[<->] (3.7,1.5) to (b);
      \tiny
      \tikzstyle{Element} = [draw, minimum width=4mm, minimum height=4mm, node distance=4mm, inner sep=0pt]
      \node[Element] [label=left:-4] at (3.6,3) {1};
      \node[Element] [label=left:-3] at (3.6,2.6) {\texttt{EE}};
      \node[Element] [label=left:-1] at (3.6,2.2) {1};
      \node[Element,fill=lightgray] [label=left:\textbf{2}] at (3.6,1.8) {\texttt{EE}};

      \node[Element] at (4.0,2.6) {0};
      \node[Element] at (4.0,2.2) {1};
      \node[Element,fill=lightgray] at (4.0,1.8) {\texttt{EE}};

      \node[Element] at (4.4,2.2) {1};
      \normalsize
      \node at (2.6,2.5) {$=$};
    \end{tikzpicture}
    \vspace{-3mm}\vspace{-3mm}
    \caption{}\label{newrowex}
  \end{figure}

\end{remark}



From the arguments that we have developed above, we can classify the covering relations in weak Bruhat order on $\sym_n^B$ in terms of permutation tableaux of type $B$ and also can prove a theorem on the inversion number of signed permutations.

We considered all possible cases in which $\sigma\lhd \sigma'$ holds, and we can conclude that
$\sigma'=\zeta(\mathcal{T}')$ covers $\sigma=\zeta(\mathcal{T})$ if and only if $\T$ and $\T'$ are in one of the relations; (WB1), (WB2), (WB3), (WB4-1), (WB4-2), (WB5-1), (WB5-2), (WB6-1), and (WB6-2) with some modification of zero column as in Remark~\ref{zero_column}. These relations can be reclassified in the following way.

\begin{theorem}\label{weakbruhatcover}
   For $\mathcal{T},\mathcal{T}'\in\mathcal{PT}_{n}^{B}$, and corresponding signed permutations $\sigma=\zeta(\mathcal{T})$, $\sigma'=\zeta(\mathcal{T}')$, $\sigma'$ covers $\sigma$ in weak Bruhat order if and only if $\T'$ is obtained by
  \begin{enumerate}
    \item adding a new box with $1$ to $\T$, or

    \item following the rule (WB3), or

    \item replacing a $1$ at which two zigzag paths from \verb"row"\,$(i)$ and \verb"row"\,$(i+1)$ (or \verb"col"\,$(j)$ and \verb"col"\,$(j+1)$) meet, by $0$, or

    \item replacing a rectangular part in \verb"row"\,$(i)$ and \verb"row"\,$(i+1)$ (or \verb"col"\,$(j)$ and \verb"col"\,$(j+1)$) with another rectangle  according to Lemma~\ref{rule}.
  \end{enumerate}
 \end{theorem}
\begin{proof} It is easy to see that (WB1) and (WB2) are of the first type, (WB4-1), (WB5-1), and (WB6-1) are of the third type and (WB4-2), (WB5-2), and (WB6-2) are of the fourth type.
\end{proof}

\begin{example}
  Let $\sigma=-2,-4,5,3,1\in\sym_{5}^{B}$ and $\T=\zeta^{-1}(\sigma)$ as in Example~\ref{exsumofalcr}. We also let $\sigma_{i}=\sigma s_{i}$ and $\T_{i}=\zeta^{-1}(\sigma_{i})$ for $i\in \{0,1,2,3,4\}$. Since the pair $(2,3)$ is not an inversion of $\sigma$, and the pairs $(1,2),(3,4),(4,5)$ are inversions of $\sigma$, $\sigma\lhd\sigma_{2}$ and $\sigma_{j}\lhd\sigma$ for $j=0,1,3,4$. The covering relations in weak Bruhat order of the corresponding permutation tableaux of type $B$ are shown in Figure~\ref{excovering}.
  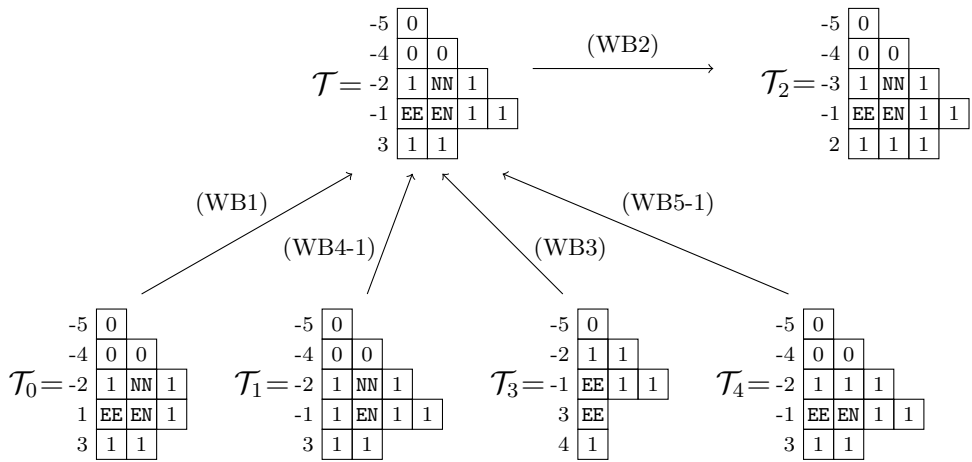
\begin{figure}[!ht]
    \begin{tikzpicture}
      \normalsize
      \node at (0.0,1.2) {$\T_{0}\hspace{-1mm}=$};
      \tikzstyle{Element} = [draw, minimum width=4mm, minimum height=4mm, inner sep=0pt]
      \tiny
      \node [Element] [label=left:-5] at (1.0,2.0) {0};
      \node [Element] [label=left:-4] at (1.0,1.6) {0};
      \node [Element] [label=left:-2] at (1.0,1.2) {1};
      \node [Element] [label=left: 1] at (1.0,0.8) {\texttt{EE}};
      \node [Element] [label=left: 3] at (1.0,0.4) {1};
      \node [Element] at (1.4,1.6) {0};
      \node [Element] at (1.4,1.2) {\texttt{NN}};
      \node [Element] at (1.4,0.8) {\texttt{EN}};
      \node [Element] at (1.4,0.4) {1};
      \node [Element] at (1.8,1.2) {1};
      \node [Element] at (1.8,0.8) {1};

      \normalsize
      \node at (3.0,1.2) {$\T_{1}\hspace{-1mm}=$};
      \tikzstyle{Element} = [draw, minimum width=4mm, minimum height=4mm, inner sep=0pt]
      \tiny
      \node [Element] [label=left:-5] at (4.0,2.0) {0};
      \node [Element] [label=left:-4] at (4.0,1.6) {0};
      \node [Element] [label=left:-2] at (4.0,1.2) {1};
      \node [Element] [label=left:-1] at (4.0,0.8) {1};
      \node [Element] [label=left: 3] at (4.0,0.4) {1};
      \node [Element] at (4.4,1.6) {0};
      \node [Element] at (4.4,1.2) {\texttt{NN}};
      \node [Element] at (4.4,0.8) {\texttt{EN}};
      \node [Element] at (4.4,0.4) {1};
      \node [Element] at (4.8,1.2) {1};
      \node [Element] at (4.8,0.8) {1};
      \node [Element] at (5.2,0.8) {1};

      \normalsize
      \node at (6.4,1.2) {$\T_{3}\hspace{-1mm}=$};
      \tikzstyle{Element} = [draw, minimum width=4mm, minimum height=4mm, inner sep=0pt]
      \tiny
      \node [Element] [label=left:-5] at (7.4,2.0) {0};
      \node [Element] [label=left:-2] at (7.4,1.6) {1};
      \node [Element] [label=left:-1] at (7.4,1.2) {\texttt{EE}};
      \node [Element] [label=left: 3] at (7.4,0.8) {\texttt{EE}};
      \node [Element] [label=left: 4] at (7.4,0.4) {1};
      \node [Element] at (7.8,1.6) {1};
      \node [Element] at (7.8,1.2) {1};
      \node [Element] at (8.2,1.2) {1};

      \normalsize
      \node at (9.4,1.2) {$\T_{4}\hspace{-1mm}=$};
      \tikzstyle{Element} = [draw, minimum width=4mm, minimum height=4mm, inner sep=0pt]
      \tiny
      \node [Element] [label=left:-5] at (10.4,2.0) {0};
      \node [Element] [label=left:-4] at (10.4,1.6) {0};
      \node [Element] [label=left:-2] at (10.4,1.2) {1};
      \node [Element] [label=left:-1] at (10.4,0.8) {\texttt{EE}};
      \node [Element] [label=left: 3] at (10.4,0.4) {1};
      \node [Element] at (10.8,1.6) {0};
      \node [Element] at (10.8,1.2) {1};
      \node [Element] at (10.8,0.8) {\texttt{EN}};
      \node [Element] at (10.8,0.4) {1};
      \node [Element] at (11.2,1.2) {1};
      \node [Element] at (11.2,0.8) {1};
      \node [Element] at (11.6,0.8) {1};

      \normalsize
      \node at (4.0,5.2) {$\T\hspace{-1mm}=$};
      \tikzstyle{Element} = [draw, minimum width=4mm, minimum height=4mm, inner sep=0pt]
      \tiny
      \node [Element] [label=left:-5] at (5.0,6.0) {0};
      \node [Element] [label=left:-4] at (5.0,5.6) {0};
      \node [Element] [label=left:-2] at (5.0,5.2) {1};
      \node [Element] [label=left:-1] at (5.0,4.8) {\texttt{EE}};
      \node [Element] [label=left: 3] at (5.0,4.4) {1};
      \node [Element] at (5.4,5.6) {0};
      \node [Element] at (5.4,5.2) {\texttt{NN}};
      \node [Element] at (5.4,4.8) {\texttt{EN}};
      \node [Element] at (5.4,4.4) {1};
      \node [Element] at (5.8,5.2) {1};
      \node [Element] at (5.8,4.8) {1};
      \node [Element] at (6.2,4.8) {1};

      \normalsize
      \node at (10.0,5.2) {$\T_{2}\hspace{-1mm}=$};
      \tikzstyle{Element} = [draw, minimum width=4mm, minimum height=4mm, inner sep=0pt]
      \tiny
      \node [Element] [label=left:-5] at (11.0,6.0) {0};
      \node [Element] [label=left:-4] at (11.0,5.6) {0};
      \node [Element] [label=left:-3] at (11.0,5.2) {1};
      \node [Element] [label=left:-1] at (11.0,4.8) {\texttt{EE}};
      \node [Element] [label=left: 2] at (11.0,4.4) {1};
      \node [Element] at (11.4,5.6) {0};
      \node [Element] at (11.4,5.2) {\texttt{NN}};
      \node [Element] at (11.4,4.8) {\texttt{EN}};
      \node [Element] at (11.4,4.4) {1};
      \node [Element] at (11.8,5.2) {1};
      \node [Element] at (11.8,4.8) {1};
      \node [Element] at (11.8,4.4) {1};
      \node [Element] at (12.2,4.8) {1};

      \scriptsize
      \draw[->] (1.4,2.4) to (4.2,4.0);
      \node at (2.6,3.6) {(WB1)};
      \draw[->] (4.4,2.4) to (5.0,4.0);
      \node at (3.9,3.0) {(WB4-1)};
      \draw[->] (7.0,2.4) to (5.4,4.0);
      \node at (7.1,3.0) {(WB3)};
      \draw[->] (10.0,2.4) to (6.2,4.0);
      \node at (8.4,3.6) {(WB5-1)};
      \draw[->] (6.6,5.4) to (9.0,5.4);
      \node at (7.8,5.7) {(WB2)};
    \end{tikzpicture}\vspace{-3mm}
  \caption{The covering relations in weak Bruhat order}\label{excovering}
  \end{figure}
\end{example}

 \begin{theorem}\label{bruhatinv}
  Let $\sigma\in \sym_n^B$ and $\mathcal{T}=\zeta^{-1}(\sigma)\in\mathcal{PT}_{n}^{B}$ then
  \begin{equation}\label{inveq}
     \ell(\sigma)=\inv(\sigma)=2\{\zero_{EE}(\mathcal{T})+\zero_{NN}(\mathcal{T})\} + \one(\mathcal{T}).
  \end{equation}
 \end{theorem}
 \begin{proof} Let $\sigma=s_{a_1}s_{a_2}\cdots s_{a_\ell}$ be a reduced expression of $\sigma$ so that $\ell = \ell(\sigma)$. Then $$\ell(s_{a_1}\cdots s_{a_{i}} s_{a_{i+1}})=\ell(s_{a_1}\cdots s_{a_{i}})+1, \mbox{ and } s_{a_1}\cdots s_{a_{i}}\lhd s_{a_1}\cdots s_{a_{i}} s_{a_{i+1}} \mbox{ for } i\in [\ell-1] \, .$$

The only signed permutation of length $0$ is the identity permutation and the permutation tableau $\T_{\varnothing}$ corresponding to the identity permutation is the one with $n$ rows and no column, having no box. Since $2\{\zero_{EE}(\mathcal{T_{\varnothing}})+\zero_{NN}(\mathcal{T_{\varnothing}})\} + \one(\mathcal{T_{\varnothing}})=0$, the equation~(\ref{inveq}) holds for the signed permutation of length $0$. Furthermore, we can check that $2\{\zero_{EE}(\T')+\zero_{NN}(\T')\}+\one(\T')=2\{\zero_{EE}(\T)+\zero_{NN}(\T)\}+\one(\T)+1$ holds in each covering relation (WB1), (WB2), (WB3), (WB4-1), (WB4-2), (WB5-1), (WB5-2), (WB6-1), and (WB6-2); where we need Lemma~\ref{rule} for (WB4-2), (WB5-2), and (WB6-2). This completes the proof.
 \end{proof}


We can rewrite the equation~(\ref{inveq}) in terms of permutation statistics due to Proposition~\ref{alignments} and Proposition~\ref{propdiagram}. We make a remark that the following equation is proved in \cite{HS-JZ} for  permutations of type $A$.

 \begin{corollary}\label{invalcr}
   Let $\sigma\in\sym_{n}^{B}$, then
   \begin{equation}
     \ell(\sigma)=\inv(\sigma)=2\al_{nest}(\sigma)+\crs(\sigma)+n-\wex(\sigma).
   \end{equation}
 \end{corollary}

%% file: BT.tex
\section{Signed permutations and bare tableaux of type $B$}\label{bt}

In this section, we prove a result on the relation between signed permutations and bare tableaux of type $B$; we prove a theorem on the number of cycles of signed permutations, which extends a theorem by A. Burstein (Theorem 4.2 in \cite{AB}) and construct $\zeta_{bare}^{-1}:\sym_{n}^{B}\rightarrow\bt_{n}^{B}$ in an explicit way.

We adopt a new way to write a cycle in $\sym_{n}^{B}$ so that we can work in more detail with cycles:
For a cycle $(c_{1},c_{2},\ldots,c_{m})$ in $\sym_{n}^{B}$ and for each $x\in [m]$, with the convention $c_{m+1}=c_1$,
\begin{itemize}
  \item if $c_{x+1}<0$ then replace $c_x$ with $|c_x|, -|c_x|$,
  \item if $c_{x+1}>0$ then replace $c_x$ with $|c_x|$,
\end{itemize}
and replace the parentheses with brackets so that we have $(c_{1},c_{2},\ldots,c_{m})= \langle a_1, a_2, \dots, a_k \rangle$ where $k\geq m$. For example, the cycle $(2,-3,-1,4)$ becomes $\langle 2,-2,3,-3,1,4\rangle$ in our new notation, and $(2)=\langle 2 \rangle$, $(-2)=\langle 2, -2 \rangle$. We call this new cycle $\langle a_1, a_2, \dots, a_k \rangle$ a \emph{path cycle}.
A nice thing about `path cycle' is that the `zigzag path' from $\verb"row"\, a_i$ or $\verb"col"\, a_i$ ends at $\verb"row"\, a_{i+1}$ or $\verb"col"\, a_{i+1}$ in the corresponding bare tableau of $\langle a_1, a_2, \dots, a_k \rangle$, with the convention $a_{k+1}=a_1$. Note that the path cycle $\langle i \rangle$ sending $i$ to $i$, for a positive integer $i$, is corresponding to the zigzag path from $\verb"row"\,i$ when $\verb"row"\,i$ is a zero row. If we replace every cycle by its path cycle in a (full) cycle notation of $\sigma\in\sym_{n}^{B}$, then we obtain a \emph{(full) path cycle notation} of $\sigma\in\sym_{n}^{B}$. Hence the $\cyc(\sigma)$ is the same as the number of path cycles in the full path cycle notation of $\sigma$. Moreover, the number of identity cycles $(i)$ where $i>0$ in the full cycle notation of $\sigma$ and the number of single path cycles $\langle i \rangle$ in the full path cycle notation of $\sigma$ are the same.

\begin{theorem}\label{cycle}
  Let $\mathcal{T}\in\mathcal{BT}_{n}^{B}$, then
  \begin{equation*}
    \cyc(\zeta_{bare}(\mathcal{T}))=\dess(\mathcal{T})+\zerorow(\mathcal{T}).
  \end{equation*}
\end{theorem}
\begin{proof}
Let $\sigma=\zeta_{bare}(\T)$. We draw a graph inside $\mathcal{T}$; $1$'s become vertices and two $1$'s are connected by an edge if they are adjacent vertices in a same row or in a same column, as it was done in \cite{JCA-AB-PN}. Note that every $1$ in a bare tableau is an essential $1$ because of the $0$-hinge condition. Hence, there can not be a cycle in the graph and we have a \emph{forest of (binary) trees}. We can observe that each tree (connected component) has a unique doubly essential $1$, which we let the \emph{root} of the tree. Thus, the number of trees in the graph is equal to the number of doubly essential $1$'s. Moreover, each positively labeled zero row in $\mathcal{T}$ is bijectively corresponding to an identity cycle $(i)$, hence, a single path cycle $\langle i \rangle$, $i\in[n]$. Therefore, it is enough to show that each tree in the graph is bijectively corresponding to a `non-single' path cycle in the full path cycle notation of $\sigma$.

We claim that each tree in the graph of a bare tableau is corresponding to a non-single path cycle in the full path cycle notation of $\sigma$, consisting of the labeling integers of the vertices in the tree through the zigzag path. We show the claim by induction on the number of vertices.

Let us first consider a tree with a single vertex, at box $(a_{1},b_{1})$; then the path cycle $\langle a_{1},b_{1} \rangle$ appears in the full path cycle notation of $\sigma$.

We now assume that the claim is true for a tree with at most $k$ vertices. Then let us consider a tree with $(k+1)$ vertices, whose root is in the box $(r_1, r_2)$. Consider the two cases; one is that the root in the box $(r_1, r_2)$ is connected with only one vertex in the box either $(r_1, r'_2)$ or $(r'_1, r_2)$, and the other is that the root is connected with two vertices in the boxes $(r_1, r'_2)$ and $(r'_1, r_2)$.

For the first case, removing the root in the box $(r_1, r_2)$ gives a subtree with $k$ vertices corresponding to a path cycle $\langle a_1, \ldots, a_s \rangle$ by the induction hypothesis. Note that the root of a tree corresponding to a path cycle $\langle c_1, \ldots, c_m \rangle$ is in the box $(\min{c_i},\max{c_i})$ because the root is the topmost and leftmost vertex in its tree and each box $(a,b)$ of a bare tableau satisfies $a<b$. Let $(a_m,a_M)$ be the box containing the root of the subtree, then we have either $r_1 = a_m$ or $r_2 = a_M$. Without loss of generality we assume that $r_1 = a_m$, then adding the vertex in the box $(r_1, r_2)$ gives an effect of compositing path cycles as $\langle a_1, \ldots, a_s \rangle \langle a_m, r_2 \rangle$ and this composition is the path cycle $\langle a_1, \ldots, a_m, r_2, a_{m+1}, \ldots, a_s \rangle$.

For the second case, removing the root in the box $(r_1, r_2)$ gives subtrees with $r$ vertices and $(k-r)$ vertices corresponding to path cycles $\langle a_1, \ldots, a_s \rangle$ and $\langle b_1, \ldots, b_t \rangle$, respectively, where $0<r<k$ by the induction hypothesis. Let $(a_m, a_M)$ and $(b_{m'}, b_{M'})$ be the boxes containing the roots of the subtrees, and we assume that the box $(a_m, a_M)$ is right of the box $(r_1, r_2)$ for convenience so we have $r_1 = a_m$ and $r_2 = b_{M'}$. Then, adding the vertex in the box $(r_1, r_2)$ gives an effect of compositing path cycles as $\langle a_1, \ldots, a_s \rangle \langle b_1, \ldots, b_t \rangle \langle a_m, b_{M'} \rangle$ and this composition is the path cycle $\langle a_{1},\ldots,a_{m},b_{M'+1},b_{M'+2},\ldots,b_{t},b_{1},\ldots,b_{M'},a_{m+1},\ldots,a_{s} \rangle$.

This proves the claim and therefore the proof of the theorem is completed.
\end{proof}

Note from the proof of Theorem~\ref{cycle}, that the root of a tree corresponding to the path cycle $\langle c_{1},\ldots,c_{u} \rangle$ is in the box $(r_{1},r_{2})$ such that $r_{1}=\min{c_{i}}$ and $r_{2}=\max{c_{i}}$. Moreover, when we make a tree by adding a root to connect two subtrees in the above proof, all labels of the subtree below (or right to) the root $(r_{1},r_{2})$ is placed to the right (or left, respectively) of  $r_{1}$ and the left (or right, respectively) of $r_{2}$ in the corresponding path cycle. From this idea, we give a construction of $\zeta_{bare}^{-1}$ in an explicit way. More precisely, we give a procedure to obtain a tree from each path cycle in the path cycle notation of a signed permutation.

Recall that the labeling set for the rows completely determines the shape of a shifted diagram. We let $\tilde{D}$ be the shifted diagram whose positive labeling set is $\wex(\sigma)$ and we define a bare tableau(filling) $\mathcal{T}$ of $\tilde{D}$ by specifying the boxes filled with $1$'s:

 \begin{itemize}
   \item Set $V=\varnothing$ and $C=\{\langle a_{1},a_{2},\ldots,a_{k} \rangle\}$.
   \item {\bf For} $\langle a_{1},\ldots,a_{l} \rangle\in C$ {\bf do}

         {\bf if} $l=1$, then {\bf let} $C:= C\setminus \{\langle a_1 \rangle\}$

        {\bf otherwise}

      \hspace{1cm}  choose $i,j\in [ l ]$ such that $a_{i}=\min\{a_1, \dots, a_l\}$ and $a_{j}=\max\{a_1, \dots, a_l\}$

      \hspace{1cm}  let $V:= V\cup \{(a_i,a_j)\}\mbox{ and, }$

      \hspace{1cm}    $C:=(C\setminus \{\langle a_1,\ldots,a_l \rangle\}) \cup \{\langle a_{i+1},\ldots,a_{j} \rangle, \langle a_1,\ldots,a_i,a_{j+1},\ldots,a_l \rangle\}$ if $i<j$,

      \hspace{1cm}   $C:=(C\setminus \{\langle a_1,\ldots,a_l \rangle\}) \cup \{\langle a_1,\ldots,a_j,a_{i+1},\ldots,a_l \rangle, \langle a_{j+1},\ldots,a_{i} \rangle\}$ if $i>j$.

      \noindent  {(end if $C=\varnothing$)}

   \item {\bf For} box $(\alpha,\beta)$ of $\tilde{D}$ {\bf do}

       {\bf if}  $(\alpha,\beta)\in V$, then fill in the  box $(\alpha,\beta)$ with $1$

       {\bf otherwise} fill in the  box $(\alpha,\beta)$ with $0$

 \end{itemize}

\begin{example}
 For $\sigma=(2,-3,-1,4)=\langle 2,-2,3,-3,1,4 \rangle\in\mathcal{BT}_{4}^{B}$, we initially let $V=\varnothing$ and $C=\{\langle 2,-2,3,-3,1,4 \rangle\}$.
 For $\langle 2,-2,3,-3,1,4 \rangle\in C$, the minimum and the maximum of $\{2,-2,3,-3,1,4\}$ are $-3$ and $4$ respectively. We, hence, have  $V=\{(-3,4)\}$ and $C=\{\langle 1,4 \rangle, \langle 2,-2,3,-3 \rangle\}$.
For $\langle 1,4 \rangle\in C$, $V$ becomes $\{(-3,4),(1,4)\}$ and $C=\{\langle 4 \rangle, \langle 1 \rangle, \langle 2,-2,3,-3 \rangle\}$.
For $\langle 4 \rangle\in C$, that is of length $1$, $C:=C\setminus \{\langle 4 \rangle\}=\{\langle 1 \rangle, \langle 2,-2,3,-3 \rangle\}$ and by considering the element $\langle 1 \rangle\in C$, $C$ becomes $\{\langle 2,-2,3,-3 \rangle\}$.
Now for $\langle 2,-2,3,-3 \rangle\in C$, since the minimum is $-3$ and the maximum is $3$, $V=\{(-3,4),(1,4),(-3,3)\}$ and $C=\{\langle 2,-2,3 \rangle, \langle -3 \rangle\}$. Then $C$ becomes $\{\langle 2,-2,3 \rangle\}$ since $\langle -3 \rangle$ is of
length $1$.
If we keep following the process then we finally have $V=\{(-3,4),(1,4),(-3,3),(-2,3),(-2,2)\}$ and $C=\varnothing$. Hence, the bare tableau $\mathcal{T}=\zeta_{bare}^{-1}(\sigma)$ of type $B$ is in Figure~\ref{BTfromperm}. One can check that $\zeta_{bare}(\mathcal{T})=\sigma$.
 \begin{figure}[!ht]
   \begin{tikzpicture}\tiny
     \tikzstyle{Element} = [draw, minimum width=4mm, minimum height=4mm, inner sep=0pt]
      \node [Element] [label=left:-4] at (1.0,2.0) {0};
      \node [Element] [label=left:-3] at (1.0,1.6) {1};
      \node [Element] [label=left:-2] at (1.0,1.2) {0};
      \node [Element] [label=left:1] at (1.0,0.8) {1};

      \node [Element] at (1.4,1.6) {1};
      \node [Element] at (1.4,1.2) {1};
      \node [Element] at (1.4,0.8) {0};

      \node [Element] at (1.8,1.2) {1};
      \node [Element] at (1.8,0.8) {0};
   \end{tikzpicture}\normalsize\vspace{-3mm}
   \caption{The bare tableau $\mathcal{T}=\zeta_{bare}^{-1}(\sigma)$}\label{BTfromperm}
 \end{figure}
\end{example}